\documentclass[11pt, reqno]{amsart}
\usepackage{amssymb,latexsym,amsmath,amsfonts,mathdots,enumitem,amsxtra}
\usepackage{latexsym}
\usepackage[mathscr]{eucal}
\usepackage{colortbl,xcolor}
\usepackage{lmodern}
\usepackage{sansmathaccent}
\usepackage[latin1]{inputenc}
\usepackage{tikz}
\usepackage{physics}
\usepackage{tikz-cd}
\usepackage{graphicx}
\usepackage{subcaption}
\usepackage{url}
\usetikzlibrary{shapes,arrows}
\usetikzlibrary{matrix,calc,shapes,arrows,positioning}
\pdfmapfile{+sansmathaccent.map}

\numberwithin{equation}{section}
\theoremstyle{plain}

\newtheorem{thm}{Theorem}[section]

\newtheorem{cor}[thm]{Corollary}
\newtheorem{lem}[thm]{Lemma}
\newtheorem{prop}[thm]{Proposition}
\theoremstyle{definition}
\newtheorem{defn}[thm]{Definition}
\newtheorem{rem}[thm]{Remark}

\numberwithin{equation}{section}
\def\beq{\begin{eqnarray}}
	\def\eeq{\end{eqnarray}}
\def\beqa{\begin{eqnarray*}}
	\def\eeqa{\end{eqnarray*}}

\def\beqn{\begin{equation}}
	\def\eeqn{\end{equation}}

\def\mg#1{}

\renewcommand{\epsilon}{\varepsilon}
\renewcommand{\phi}{\varphi}

\renewcommand{\bf}[1]{\textbf{#1}}
\renewcommand{\it}[1]{\textit{#1}}

\renewcommand{\sf}[1]{\textsf{#1}}

\numberwithin{equation}{section}%Code for numbering equations sectionwise
\allowdisplaybreaks[4] %If you prefer a strategy of letting page breaks fall where they may, even in the middle of a multi-line equation, then you might put \allowdisplaybreaks[1] in the preamble of your document. An optional argument 1ÃÂ¢ÃÂÃÂ4 can be used for finer control: [1] means allow page breaks, but avoid them as much as possible; values of 2,3,4 mean increasing permissiveness. When display breaks are enabled with \allowdisplaybreaks, the \\* command can be used to prohibit a pagebreak after a given line, as usual.
\setlist[enumerate]{font=\upshape,noitemsep, topsep=0pt} % while enumerating the numbering will be in up-shape. Default is italics. Also reduce item seperation space.
\setlist[itemize]{noitemsep, topsep=0pt}

\begin{document}
	
\title[Functional Models]{Functional models for $\Gamma_n$-contractions}
\author{Shubhankar Mandal, Avijit Pal and Bhaskar Paul}
	
\address[S. Mandal]{Department of Mathematics, IIT Bhilai, 6th Lane Road, Jevra, Chhattisgarh 491002}
\email{S. Mandal:shubhankarm@iitbhilai.ac.in}

\address[A. Pal]{Department of Mathematics, IIT Bhilai, 6th Lane Road, Jevra, Chhattisgarh 491002}
\email{A. Pal:avijit@iitbhilai.ac.in}

\address[B. Paul]{Department of Mathematics, IIT Bhilai, 6th Lane Road, Jevra, Chhattisgarh 491002}
\email{B. Paul:bhaskarpaul@iitbhilai.ac.in }


	\maketitle
	
\begin{abstract}
		This article develops several functional models for a given $\Gamma_n$-contraction. The first model is motivated by the Douglas functional model for a contraction. We then establish factorization results that  clarify the relationship between a minimal isometric dilation and an arbitrary isometric dilation of a contraction. Using these factorization results, we obtain a Sz.-Nagy-Foias type functional model for a completely non-unitary $\Gamma_n$-contraction, as well as  Sch\"affer type functional model for $\Gamma_n$-contraction.
\end{abstract}
	
\section{Introduction and preliminaries}

For $n \geq 2$, define the symmetrization map 
$\textbf{s} : \mathbb{C}^n \to \mathbb{C}^n$ 
by
\[
\textbf{s}(\textbf{z}) = (s_1(\textbf{z}), \ldots, s_n(\textbf{z})), \quad \textbf{z} = (z_1, \ldots, z_n) \in \mathbb{C}^n,
\]
where, for $1\leq i \leq n,$
\[
s_i(\textbf{z}) = \sum_{1 \leq k_1 < \cdots < k_i \leq n} z_{k_1} \cdots z_{k_i}, 
\quad s_0 = 1.
\]
A closed symmetrized polydisc is defined by $\Gamma_n := \textbf{s}(\bar{\mathbb{D}}^n)$. The map $\textbf{s}$ is a proper holomorphic map (see \cite{Rudin}). Although  $\Gamma_n$ is not convex, it is polynomially convex. The corresponding open symmetrized polydisc is  $\mathbb{G}_n := \textbf{s}(\mathbb{D}^n)$, and its distinguished boundary $b\Gamma_n$ is given by $\textbf{s}(\mathbb{T}^n)$.

Assume that $\Omega$ be a compact set in $\mathbb{C}^n$, and let $\mathcal{O}(\Omega)$ denote the algebra of bounded holomorphic functions in some neighbourhood of $\Omega$ equipped with the sup norm
$$\| f \|_{\infty,\Omega} = \sup_{z \in \Omega} \{|f(z)|:~f\in \mathcal{O}(\Omega)\}.$$ Let $\textbf{T}=(T_1,\dots,T_n)$ be a commuting $n$-tuple of bounded operators on a Hilbert space $\mathcal H,$ and denote its joint spectrum  by  $\sigma(\mathbf T)$. Define a map $$\rho_{\mathbf T}:\mathcal O(\Omega)\rightarrow\mathcal B(\mathcal H)$$ by $$1\to I~{\rm{and}}~z_i\to T_i~{\rm{for}}~1\leq i\leq n. $$ It is immediate that $\rho_{\mathbf T}$ is a homomorphism. A compact set $\Omega \subset \mathbb{C}^n$ is called a \emph{spectral set} for $\textbf{T} $, if the joint spectrum 
$\sigma(\textbf{T}) \subseteq \Omega$ 
and the homomorphism $\rho_{\textbf{T}}: \mathcal O(\Omega) \to \mathcal B(\mathcal{H})$ is contractive. The closed unit disc plays a central role in operator theory: it is a spectral set for every contraction on a Hilbert space $\mathcal H,$  as established by the following theorem (see [Corollary 1.2, \cite{paulsen}]).

\begin{thm} 
Let $T\in \mathcal B(\mathcal H)$ be a contraction. Then
$$\|p(T)\|\leq \|p\|_{\infty, \bar {\mathbb D}}:=\sup\{|p(z)|: |z|\leq1\} $$ for every polynomial $p.$
\end{thm}
The following theorem presented here is a refined version of the Sz.-Nagy dilation theorem [Theorem 1.1, \cite{paulsen}].
\begin{thm} Let $T\in \mathcal B(\mathcal H)$ be a contraction. Then there exists a larger Hilbert space $\mathcal K$ that contains $\mathcal H$ as a subspace, and a unitary operator $U$ acting on a Hilbert space $\mathcal K \supseteq \mathcal H$ with the property that $\mathcal K$ is the smallest closed reducing subspace for $U$ containing $\mathcal H$ such that
$$P_\mathcal H\,U^n|_{\mathcal H}=T^n, ~{\rm{ for ~all}} ~n\in \mathbb N\cup \{0\}.$$
\end{thm}

Schaffer constructed  a unitary dilation for a contraction $T$. The spectral theorem for unitary operators and the existence of power dilation for contractions ensure the validity of the von Neumann inequality. Let $\Omega$ be a compact subset of $\mathbb C ^n. $ Let $F=\left(\!(f_{ij})\!\right)$ be a matrix-valued polynomial defined on $\Omega.$ We say that $\Omega$ is a complete spectral set (complete $\Omega$-contraction) for $\mathbf T$ if $\|F(\mathbf T) \| \leq \|F\|_{\infty, \Omega}$ for every $F\in \mathcal O(\Omega)\otimes \mathcal M_{k\times k}(\mathbb C), k\geq 1$.  The notion, ``property $P$" introduced in \cite{Bagchi} for a particular class of operators and was used to study certain properties of finite-dimensional Banach spaces. In this paper, we say that 
 $\Omega$ has property $P$ if the following holds: $\Omega$ is a spectral set for a commuting $m$-tuple of operators $\mathbf{T}$, if and only if it is a complete spectral set for $\mathbf{T}$.
In either case, property $P$ guarantees the existence of a dilation, which we describe using Arveson's framework.
 
 A commuting $n$-tuples of operators $\mathbf{T}$ for which $\Omega$ is a spectral set is said to admit a $\partial \Omega$ normal  dilation if there exists a Hilbert space $\mathcal K$ containing $\mathcal H$ as a subspace and a commuting $n$-tuples of normal operators $\mathbf{N}=(N_1,\ldots,N_n)~{\rm{ on }}~\mathcal K~{\rm{with}}~\sigma(N)\subseteq \partial \Omega$ such that
$$P_{\mathcal H}F(\mathbf N)|_{\mathcal H}=F(\mathbf T) ~{\rm{for~ all~}} F\in \mathcal O(\Omega).$$
In 1969, Arveson \cite{A} showed that a commuting $n$-tuple of operators $\mathbf{T}$ having $\Omega$ as a spectral set for $\mathbf{T}$ admits a $\partial \Omega$ normal dilation if and only if it satisfies the property $P.$  Subsequently, J. Agler \cite{agler} proved that the annulus has the property $P$. In contrast, M. Dristchell and S. McCullough \cite{michel} showed that no planar domain of connectivity $n \geq 2$ can have property $P.$  In the multi-variable setting, both the symmetrized bi-disc and the bi-disc are known to have  property $P$, as shown by Agler and Young \cite{young} and by Ando \cite{paulsen}, respectively. The first counterexample in higher dimensions was given by Parrott \cite{paulsen} for $\mathbb D^n$ when $n > 2. $ Further negative results were obtained by G. Misra \cite{GM,sastry}, V. Paulsen \cite{vern}, and E. Ricard \cite{pisier}, who showed that a ball in $\mathbb{C}^m$, with respect to some norm $\|\cdot\|_{\Omega}$, cannot have property $P$ when  $m \geq 3$. More recently, it was shown in \cite{cv} that if $B_1$ and $B_2$ are not simultaneously diagonalizable via unitary, then the domain $ \Omega_{\mathbf B}:=   \{(z_1
  ,z_2) :\|z_1 B_1 + z_2 B_2 \|_{\rm op} < 1\}, $ does not have property $P,$ where $\mathbf B=(B_1, B_2)$ in $\mathbb C^2 \otimes \mathcal
  M_2(\mathbb C)$ and $B_1$ and $B_2$ are independent.

 If $\Gamma_n$ is a spectral set, then a commuting $n$-tuple $(S_1, \ldots, S_n)$ is called a $\Gamma_n$-contraction. It is evident from \cite[Proposition 4.1]{19} that if $(S_1,\dots, S_n)$ is a $\Gamma_n$-contraction, then $\|S_i\|\leq \binom{n-1}{i}+ \binom{n-1}{n-i}$ for $i=1,\dots,n-1$ and $\|S_n\|\leq 1$. This gives an estimate for the norm of $S_i$. Since $S_n$ is a contraction, we define the defect operators $$D_{S_n^*}=(I-S_n S_n^*)^{\frac{1}{2}}, \quad \text{and} \quad D_{S_n}=(I-S^*_n S_n)^{\frac{1}{2}}$$ corresponding to the defect spaces $$\mathcal{D} _{S_n^*}=\overline{\text{Ran}}D_{S_n^*}, \quad \text{and} \quad \mathcal{D} _{S_n}=\overline{\text{Ran}}D_{S_n},$$ respectively. The notions of $\Gamma_n$-unitary, $\Gamma_n$-isometry, and pure $\Gamma_n$-isometry are reviewed in \cite{sbs}.
\begin{defn}\label{def:Gamma_n}
	Let $(S_1,\dots,S_n)$ be a commuting $n$-tuple of bounded operators on a Hilbert space $\mathcal{H}$. We say that $(S_1,\dots,S_n)$ is
	\begin{enumerate}
		\item a $\Gamma_n$-\emph{unitary} if the joint spectrum $\sigma(S_1,\dots,S_n)\subseteq b\Gamma_n$ and each $S_i$ is normal on $\mathcal{H}$;
		\item a $\Gamma_n$-\emph{isometry} if there exists a Hilbert space $\mathcal{K}\supseteq \mathcal{H}$ and a $\Gamma_n$-unitary $(\widetilde S_1,\dots,\widetilde S_n)$ on $\mathcal{K}$ such that $\mathcal{H}$ is a common invariant subspace for $\widetilde S_1,\dots,\widetilde S_n$ and $S_i=\widetilde S_i|_{\mathcal{H}}$ for $i=1,\dots,n$;
		\item a $\Gamma_n$-\emph{co-isometry} if $(S_1^*,\dots,S_n^*)$ is a $\Gamma_n$-isometry;
		\item a \emph{pure} $\Gamma_n$-\emph{isometry} if $(S_1,\dots,S_n)$ is a $\Gamma_n$-isometry and $S_n$ is a pure isometry.
	\end{enumerate}
\end{defn}
The following theorem gives us the characterization of the $\Gamma_n$-isometry and $\Gamma_n$-unitary (see \cite{sbs}).
\begin{thm}\label{thm1.4}
	Suppose $S_1,\ldots,S_n$ are commuting operators acting on a Hilbert space $\mathcal H$ and $\gamma_i = \frac{n-i}{n},~ \text{for}~i = 1,2,\ldots,n-1.$
	Then the following are equivalent:
	\begin{enumerate}
		\item $(S_1,\ldots,S_n)$ is a $\Gamma_n$-isometry (respectively unitary);
		\item $(\gamma_1 S_1,\ldots,\gamma_{n-1} S_{n-1})$ is a $\Gamma_{n-1}$-contraction, $S_i = S_{n-i}^* S_n$ for $i=1,\ldots,n-1$, and $S_n$ is an isometry (respectively unitary);
		\item $(S_1,\ldots,S_n)$ is a $\Gamma_n$-contraction and $S_n$ is an isometry (respectively unitary).
	\end{enumerate}
\end{thm}

\begin{prop}\label{eg1.5}
	Let $(M_{\Phi_1},\dots,M_{\Phi_{n-1}},M_z)$ be a tuple of operators on $H_{\mathcal{E}}^2(\mathbb{D})$, where $\Phi_i(z)=E_i^*+zE_{n-i}$ and $E_i,E_{n-i}$ be operators on $\mathcal{E}$. Then the tuple $(M_{\Phi_1},\dots,M_{\Phi_{n-1}},M_z)$ is commuting if and only if the following condition holds 
	\begin{equation}\label{eq1.1}
		[E_i,E_j]=0~~\text{and}~~[E_i^*,E_{n-j}]=[E_j^*,E_{n-i}],
	\end{equation}
	for every $i,j=1,\dots,n-1$, where $[A,B]=AB-BA$ denotes the commutator. Additionally, it is easy to verify that $M_{\Phi_i}=M_{\Phi_{n-i}}^*M_z$ for all $i=0,\dots,n-1$. We refer to the condition mentioned in \eqref{eq1.1} as the \textit{commutativity condition}.
\end{prop}

In operator theory, the solution of operator equations often plays a central and essential role. In this context, Bhattacharyya, Pal, and Roy (see \cite[Lemma 4.1 and Theorem 4.2]{Roy}) established the existence and uniqueness of a solution to an operator equation associated with a $\Gamma$-contraction, which they termed the fundamental equation of the $\Gamma$-contraction.
The following theorem extends this result by establishing the existence and uniqueness of solutions to the fundamental equations associated with a $\Gamma_n$-contraction (see \cite{14}).
\begin{thm}\label{thm1.2}
	\textit{(Existence and Uniqueness)} For $n\geq2$, let $(S_1,\dots, S_n)$ be a $\Gamma_n$-contraction on a Hilbert space $\mathcal{H}$. Then there are unique operators $F_1,\dots, F_{n-1}\in \mathcal{B}(\mathcal{D}_{S_n})$ such that $S_i-S^*_{n-i}S_n=D_{S_n}F_iD_{S_n}$ and $S_{n-i}-S^*_i S_n=D_{S_n}F_{n-i}D_{Sn}$, for $i =1,\dots,n-1$. Moreover, $\omega(F_i+F_{n-i}z)\leq \binom{n-1}{i}+ \binom{n-1}{n-i}$ for all $z\in\mathbb{T}$.
\end{thm}
The next result gives the different
characterization of the fundamental operators $E_i$'s of a $\Gamma_n$-contraction $(S_1,\dots, S_n)$
which is described in \cite[Proposition 2.5]{14}.
\begin{thm}\label{thm1.3}
	For $i=1,\dots,n-1$, the fundamental operators $E_i$ of a $\Gamma_n$-contraction $(S_1,\dots, S_n)$ is the unique bounded linear operator on $\mathcal D_{S_n}$ that satisfies the following operator
	equation
	$$D_{S_n} S_i = F_i D_{S_n}+ F^*_{n-i} D_{S_n} S_n.$$
\end{thm}
 Assume that $T\in \mathcal{B}(\mathcal{H})$ is a contraction and $V\in \mathcal{B}(\mathcal{K})$ is an isometry. Recall that if $V$ is an isometric dilation of $T$, then there is an isometry $\mathscr{E}:\mathcal{H}\rightarrow\mathcal{K}$ such that
\begin{equation}\label{iso-dil}
	\mathscr{E}T^*=V^*\mathscr{E}.
\end{equation}
Conversely, if there is an isometry $\mathscr{E}:\mathcal{H}\rightarrow\mathcal{K}$
that satisfies equation \eqref{iso-dil} then $V$ is an isometric dilation of $T$. Moreover, the isometric dilation is minimal if $$\mathcal{K}=\overline{\operatorname{Span}}\{V^m(\mathscr{E}\mathcal{H}):~m\geq0\}.$$  
Although explicit constructions of the minimal isometric dilation of a single contraction are well established-most notably the Schaffer construction (see \cite{16}), the Douglas construction (see \cite{10}), and the Sz.-Nagy minimal dilation for completely non-unitary contractions (see \cite{Nagy}). 
In this article, we focus on the explicit construction of minimal isometric dilations for $\Gamma_n$-contractions and subsequently develop the associated functional models.

We denote the unit circle by $\mathbb T.$  Let $\mathcal E$  be a separable Hilbert space, and  $\mathcal B(\mathcal E)$ denotes the space of bounded linear operators on $\mathcal E$ equipped with the operator norm. Let $H^2_{\mathcal E}(\mathbb{D})$ denote  the  Hardy space of analytic $\mathcal E$-valued functions defined on the unit disk  $\mathbb D$. Let $ L^2_{\mathcal E}(\mathbb{T})$ represent the Hilbert space of square-integrable $\mathcal E$-valued functions on the unit circle $\mathbb T,$ equipped with the natural inner product. The space $H^{\infty}(\mathcal B(\mathcal E))$ consists of bounded analytic $\mathcal B(\mathcal E)$-valued functions defined on $\mathbb D$. Let $L^{\infty}(\mathcal B(\mathcal E))$ denote the space of bounded measurable $\mathcal B(\mathcal E)$-valued functions on $\mathbb T$.  For $\phi \in L^{\infty}(\mathcal B(\mathcal E)),$ the Toeplitz operator associated with the symbol  $\phi$ is denoted by $T_{\phi}$ and is defined as follows: 
 $$T_{\phi}f=P_{+}(\phi f), \quad f \in H^2_{\mathcal E}(\mathbb{D}),$$ where $P_{+} : L^2_{\mathcal E}(\mathbb{D}) \to H^2_{\mathcal E}(\mathbb{D})$ is the orthogonal projecton.  In particular, $T_z$ is the
 unilateral shift operator $M_z$ on $H^2_{\mathcal E}(\mathbb{D})$  and $T_{\bar{z}}$ is the backward shift $M_z^*$ on $H^2_{\mathcal E}(\mathbb{D})$. Note that the vector valued Hardy space $H^2_{\mathcal{E}}(\mathbb{D})$ can be identified as $H^2(\mathbb{D})\otimes \mathcal{E}$ via the unitary operator $U:H^2_{\mathcal{E}}(\mathbb{D})\rightarrow H^2(\mathbb{D})\otimes \mathcal{E}$ defined by $$z^m e\mapsto z^m \otimes e,\quad \text{for}~e\in \mathcal{E},~m \in \{0\}\cup \mathbb{N}.$$

% For a  $\Gamma_n$-contraction $(S_1^*,\dots, S_n^*)$, Theorem \ref{thm1.2} implies that there exists a unique solution $(X_1,\dots,X_{n-1})=(E_1,\dots,E_{n-1})$ to the system of equations
%\begin{equation*}
%	S_i^*-S_{n-i}S_n^*=D_{S_n^*}E_iD_{S_n^*}~\text{and}~S_{n-i}^*-S_i S_n^*=D_{S_n^*}E_{n-i}D_{S_n^*},
%\end{equation*}
%for $i =1,\dots,n-1$. Equivalently, by Theorem \ref{thm1.3}, the fundamental operators $(E_1, \dots, E_{n-1})$ provide the unique solution to the alternative system:
%\begin{equation}\label{eq3.20}
%	D_{S_n^*} S_i^* = X_i D_{S_n^*}+ X^*_{n-i} D_{S_n^*} S_n^*,
%\end{equation}
%for $i =1,\dots,n-1$.

The paper is organized as follows: In Section 2, using the von Neumann-Wold decomposition for a $\Gamma_n$-isometry, we constructed the canonical $\Gamma_n$-unitary associated with a $\Gamma_n$-contraction. This article then develops several functional models for $\Gamma_n$-contraction.  Motivated by the Douglas functional model for a single contraction, we construct  functional models for $\Gamma_n$-contraction  in Section 3.
In Section 4, we establish factorization results that clarify the relationship between a minimal isometric dilation and an arbitrary isometric dilation of a given contraction. Finally, building on these factorization results, Section 5 develops a Sz.-Nagy-Foias type functional model for completely non-unitary $\Gamma_n$-contractions, as well as a Schaffer-type functional model for $\Gamma_n$ -contractions.

\section{Canonical Construction of $\Gamma_n$-Unitary from $\Gamma_n$-Contraction}
We begin this section by establishing several characterizations of $\Gamma_n$-contraction. These results are stated in the form of lemmas, as they will play a crucial role in the development of a Douglas type functional model for $\Gamma_n$-contraction.
\begin{lem}\label{lem2.1}
	Let $(M_1, \dots, M_n)$ be a $\Gamma_n$-contraction. Then the following assertions hold:
	
	\begin{enumerate}
		\item Let $(M_1, \dots, M_n)$ be a tuple of operator acting on the Hilbert space $\mathcal{K}$, and let $(B_1, \dots, B_n)$ be a tuple of operator acting on the Hilbert space $\mathcal{H}$. If there exists a unitary operator $U : \mathcal{K} \to \mathcal{H}$ such that
		$$B_i = U M_i U^* \quad \text{for each } i = 1, \dots, n,$$
		then $(B_1, \dots, B_n)$ is also a $\Gamma_n$-contraction.
		
		\item Let $\mathcal{H}$ and $\mathcal{K}$ be Hilbert spaces, and for each $i = 1, \dots, n$, let
		$$M_i = 
		\begin{bmatrix}
			A_i & 0 \\
			0 & B_i
		\end{bmatrix}
		\quad \text{acting on} \quad \mathcal{H} \oplus \mathcal{K},$$
		where $A_i \in \mathcal{B}(\mathcal{H})$ and $B_i \in \mathcal{B}(\mathcal{K})$. Then the tuple $(A_1, \dots, A_n)$, and $(B_1, \dots, B_n)$ are $\Gamma_n$-contractions.
	\end{enumerate}
\end{lem}
\begin{proof}
	Let $F$ be a $n$-variable polynomial. Also, if $(M_1,\dots M_n)$ is unitary equivalent to $(B_1,\dots,B_n)$ via unitary operator $U:\mathcal{K}\rightarrow \mathcal{H}$ then using hypothesis, we have 
	\begin{align*}
		\|F(B_1,\dots,B_n)\|_{\mathcal{H}}&=\|F(UM_1U^*,\dots,UM_nU^*)\|_{\mathcal{H}}\\
		&=\|UF(M_1,\dots,M_n)U^*\|_{\mathcal{H}}\\&\leq \|F(M_1,\dots M_n)\|_{\mathcal{K}}\\
		&\leq \sup_{z \in \Gamma_n}|F(z)|.
	\end{align*}
	This completes the proof of assertion $(1)$. Moreover, by hypothesis
	\begin{align*}
		\|F(A_1,\dots,A_n)\|_{\mathcal{H}}&=\|F(M_1|_{\mathcal{H}},\dots,M_n|_{\mathcal{H}})\|_{\mathcal{H}}\\&\leq \|F(M_1,\dots M_n)\|_{\mathcal{H}\oplus\mathcal{K}}\\
		&\leq \sup_{z \in \Gamma_n}|F(z)|.
	\end{align*}
	Similar argument holds for the $\Gamma_n$-contraction $(B_1, \dots, B_n)$. This completes the proof of assertion $(2)$. 
\end{proof}
\begin{lem}\label{lem2.2}
	Let $\mathcal{H}$ and $\mathcal{K}$ be Hilbert spaces, and for each $i = 1, \dots, n$, let
	$$M_i = 
	\begin{bmatrix}
		A_i & 0 \\
		0 & B_i
	\end{bmatrix}
	\quad \text{on } \mathcal{H} \oplus \mathcal{K},
	$$
	where $A_i \in \mathcal{B}(\mathcal{H})$ and $B_i \in \mathcal{B}(\mathcal{K})$. Suppose that the tuples $(A_1, \dots, A_n)$ and $(B_1, \dots, B_n)$ are $\Gamma_n$-contractions (respectively, $\Gamma_n$-isometries or $\Gamma_n$-unitaries). Then the tuple $(M_1, \dots, M_n)$ is a $\Gamma_n$-contraction (respectively, $\Gamma_n$-isometry or $\Gamma_n$-unitary).
\end{lem}
\begin{proof}
	Let $F$ be a $n$-variable polynomial. Since, the tuple $(A_1, \dots, A_n)$ acting on $\mathcal{H}$ and $(B_1, \dots, B_n)$ acting on $\mathcal{K}$ are $\Gamma_n$-contractions. It follows that, 
	\begin{equation*}\label{eq2.2}
		\|F(A_1, \dots, A_n)\|_{\mathcal{H}}\leq \sup_{z \in \Gamma_n}|F(z)|~\text{and}~\|F(B_1, \dots, B_n)\|_{\mathcal{K}}\leq \sup_{z \in \Gamma_n}|F(z)|.
	\end{equation*}  Now, using the above equation, we have, 
	\begin{align*}
		\|F(M_1,\dots,M_n)\|_{\mathcal{H}\oplus \mathcal{K}}&=\left\|\begin{bmatrix}
			F(A_1, \dots, A_n) & 0\\
			0 & F(B_1, \dots, B_n)
		\end{bmatrix}\right\|_{\mathcal{H} \oplus \mathcal{K}}\\&= \max\{\|F(A_1, \dots, A_n)\|_{\mathcal{H}},\|F(B_1, \dots, B_n)\|_{\mathcal{K}}\} \\
		&\leq \sup_{z \in \Gamma_n}|F(z)|.
	\end{align*}
	It proves that the tuple $(M_1,\dots,M_n)$ is a $\Gamma_n$-contraction. Furthermore, if the tuples $(A_1, \dots, A_n)$ and $(B_1, \dots, B_n)$ are $\Gamma_n$-isometries (respectively, $\Gamma_n$-unitaries), then both $(A_1, \dots, A_n)$ and $(B_1, \dots, B_n)$ are $\Gamma_n$-contractions. Consequently, $(M_1,\dots,M_n)$ is also a $\Gamma_n$-contraction. Since $A_n$ and $B_n$ are isometries (respectively, unitaries), it follows that $M_n$ is an isometry (respectively, unitary). Hence, by Theorem \ref{thm1.4}, the tuple $(M_1, \dots, M_n)$ is a $\Gamma_n$-isometry (respectively,  $\Gamma_n$-unitary). This completes the proof.
\end{proof}
The von Neumann-Wold decomposition asserts that every isometry decomposes into a direct sum of a pure isometry (a shift of appropriate multiplicity) and a unitary operator. The following result not only provides a model for an arbitrary $\Gamma_n$-isometry but also generalises the von Neumann-Wold decomposition to the setting of $\Gamma_n$-isometry.
\begin{thm}\label{thm2.3}
	Let $(V_1,\ldots,V_n)$ be the tuple of operator acting on a Hilbert space $\mathcal H$. Then, the tuple $(V_1,\ldots,V_n)$ is a $\Gamma_n$-isometry if and only if the tuple $(\gamma_1V_1,\dots,\gamma_{n-1}V_{n-1})$ is a $\Gamma_{n-1}$-contraction, there exist Hilbert spaces $\mathcal{E}$ and $\mathcal{F}$, operators $E_1,\dots,E_{n-1}$ on $\mathcal{E}$ satisfy the commutativity condition in \eqref{eq1.1} and a $\Gamma_n$-unitary $(U_1,\dots,U_n)$ on $\mathcal{F}$ such that the tuple $(V_1,\ldots,V_n)$ is unitarily equivalent to the operator tuple
	\begin{equation*}\label{eq2.2}
		\left(\begin{bmatrix}
			M_{\Phi_1} & 0\\
			0 & U_1
		\end{bmatrix},\dots,\begin{bmatrix}
			M_{\Phi_{n-1}} & 0\\
			0 & U_{n-1}
		\end{bmatrix}, \begin{bmatrix}
			M_z & 0\\
			0 & U_n
		\end{bmatrix}\right),
	\end{equation*}
	where $M_{\Phi_i}$ is the multiplication operator on $H_{\mathcal{E}}^2(\mathbb{D})$ corresponding to  $\Phi_i(z)=E_i^*+zE_{n-i} \in H^\infty (\mathcal{B}(\mathcal{E}))$, and $M_z$ denotes the pure isometry.
\end{thm}
\begin{proof}
	Assume that $(V_1,\ldots,V_n)$ is a $\Gamma_n$-isometry. By Theorem \ref{thm1.4}, the tuple $(V_1,\ldots,V_n)$ is a $\Gamma_n$-contraction, and $V_n$ is an isometry. If $(V_1,\ldots,V_n)$ is a $\Gamma_n$-contraction, then  $(\gamma_1V_1,\dots,\gamma_{n-1}V_{n-1})$ is a $\Gamma_{n-1}$-contraction. Applying the Wold decomposition to the isometry $V_n$, we get: there exist Hilbert spaces $\mathcal E$ and $\mathcal F$ along with a unitary operator $U:\mathcal{H}\rightarrow H_{\mathcal{E}}^2(\mathbb{D})\oplus \mathcal{F}
	$ such that 
	\begin{equation}\label{eq2.3}
		UV_nU^*=\begin{bmatrix}
			M_z & 0\\
			0 & U_n
		\end{bmatrix}:H_{\mathcal{E}}^2(\mathbb{D})\oplus \mathcal{F}\rightarrow H_{\mathcal{E}}^2(\mathbb{D})\oplus \mathcal{F},
	\end{equation} for some unitary $U_n$ on $\mathcal F$. We claim that the same unitary operator $U$ appearing in equation \eqref{eq2.3} satisfies the following relation
	\begin{equation}\label{eq2.4}
		U(V_1, \dots, V_{n-1})U^*=\left(\begin{bmatrix}
			M_{\Phi_1} & 0\\
			0 & U_1
		\end{bmatrix},\dots,\begin{bmatrix}
			M_{\Phi_{n-1}} & 0\\
			0 & U_{n-1}
		\end{bmatrix}\right),
	\end{equation}
	where each $M_{\Phi_i}$ denotes a multiplication operator on $H_{\mathcal{E}}^2(\mathbb{D})$, and $U_i \in \mathcal{B}(\mathcal{F})$ for $i=1,\dots,n-1$. To establish this, suppose more generally that 
	$$U(V_1, \dots, V_{n-1})U^*=\left(\begin{bmatrix}
		V_{11}^{(1)} & V_{12}^{(1)}\\
		V_{21}^{(1)} & U_1
	\end{bmatrix},\dots, \begin{bmatrix}
		V_{11}^{(n-1)} & V_{12}^{(n-1)}\\
		V_{21}^{(n-1)} & U_{n-1}
	\end{bmatrix} \right).$$ 
	Since $(V_1,\ldots,V_n)$ is a $\Gamma_n$-isometry, the operators $V_i$ and $V_j$ commute for all $i,j=1,\dots,n$. Consequently,
	\begin{equation}\label{eq2.6}
		UV_iU^*UV_jU^*=UV_jU^*UV_iU^*.
	\end{equation}  Now, by setting $j=n$ and for $i=1,\dots,n-1$ in equation \eqref{eq2.6}, we obtain $$\begin{bmatrix}
		V_{11}^{(i)} & V_{12}^{(i)}\\
		V_{21}^{(i)} & U_i
	\end{bmatrix}\begin{bmatrix}
		M_z & 0\\
		0 & U_n
	\end{bmatrix}=\begin{bmatrix}
		M_z & 0\\
		0 & U_n
	\end{bmatrix} \begin{bmatrix}
		V_{11}^{(i)} & V_{12}^{(i)}\\
		V_{21}^{(i)} & U_i
	\end{bmatrix}.$$ It is equivalent to  $$\begin{bmatrix}
		V_{11}^{(i)}M_z & V_{12}^{(i)}U_n \\
		V_{21}^{(i)}M_z & U_iU_n
	\end{bmatrix}=\begin{bmatrix}
		M_zV_{11}^{(i)} & M_zV_{12}^{(i)} \\
		U_nV_{21}^{(i)} & U_nU_i
	\end{bmatrix},$$ for all $i=1,\dots,n-1$. Comparing the $(1,2)$-th entry of the above equation, we have $V_{12}^{(i)}U_n=M_zV_{12}^{(i)}$ for all $i=1,\dots,n-1$. By \cite[Lemma 2.5]{ay1}, it follows that $V_{12}^{(i)}=0$ for all $i=1,\dots,n-1.$ Again using the characterization of $\Gamma_n$-isometry, it is evident that  $V_i=V^*_{n-i}V_n$ for all $i=1,\dots,n-1$. Hence, 
	\begin{equation*}\label{2.7}
		UV_iU^*=UV_{n-i}^*U^*UV_nU^*,
	\end{equation*} for all $i=1,\dots,n-1$.  Substituting the value we get  $$\begin{bmatrix}
		V_{11}^{(i)} & 0\\
		V_{21}^{(i)} & U_i
	\end{bmatrix}=\begin{bmatrix}
		{V_{11}^{(n-i)}}^* & {V_{21}^{(n-i)}}^*\\
		0 & U_{n-i}^*
	\end{bmatrix}\begin{bmatrix}
		M_z & 0\\
		0 & U_n
	\end{bmatrix},$$ for all $i=1,\dots,n-1$. The above equation can be expressed equivalently as 
	\begin{equation}\label{matrix-eq1}
		\begin{bmatrix}
			V_{11}^{(i)} & 0\\
			V_{21}^{(i)} & U_i
		\end{bmatrix}=\begin{bmatrix}
			{V_{11}^{(n-i)}}^*M_z & {V_{21}^{(n-i)}}^*U_n\\
			0 & U_{n-i}^*U_i
		\end{bmatrix},
	\end{equation} for all $i=1,\dots,n-1$. Compare the $(2,1)$-th entry of equation \ref{matrix-eq1} to get $V^{(i)}_{21}=0$ for all $i=1,\dots,n-1.$ Therefore, we are left with the following operator equations
	\begin{equation}\label{eq2.7}
		V^{(i)}_{11}M_z=M_zV^{(i)}_{11}, \quad \text{for} \quad i=1,\dots,n-1,
	\end{equation}
	and
	\begin{equation}\label{matrix-eq2}
		V^{(i)}_{11}={V_{11}^{(n-i)}}^*M_z,\quad \text{for} \quad i=1,\dots,n-1.
	\end{equation}
	Observe that each operator $V_{11}^{(i)}\in \mathcal{B}(H_{\mathcal{E}}^2(\mathbb{D}))$. By equation \eqref{eq2.7}, there exists a bounded analytic function $\Phi_i:\mathbb{D}\rightarrow \mathcal{B}(\mathcal E)$ such that $V_{11}^{(i)}=M_{\Phi_i}$ for all $i=i,\dots,n-1$, where $M_{\Phi_i}$ denotes the multiplication operator on $H_{\mathcal{E}}^2(\mathbb{D})$. Therefore,
	\begin{equation}\label{eq2.9}
		M_{\Phi_i}=M_{\Phi_{n-i}}^*M_z,
	\end{equation} for all $i=1,\dots,n-1$. Now, inserting the power series expansion of $\Phi_i$ in equation \eqref{eq2.9}, implies that $\Phi_i(z)=E_i^*+zE_{n-i}$, for some operator $E_i, E_{n-i} \in \mathcal{B}(\mathcal{E})$. This establishes the assertion in equation \eqref{eq2.4} that the off-diagonal blocks $V_{12}^{(i)}$ and $V_{21}^{(i)}$ are equal to $0$, and that $V_{11}^{(i)} = M_{\Phi_i}$ for each $i = 1, \dots, n-1$. Hence, using equation \eqref{eq2.3} and \eqref{eq2.4}, we arrive at the following relation 
	\begin{equation*}
		U(V_1,\ldots,V_{n-1},V_n)U^*=\left(\begin{bmatrix}
			M_{\Phi_1} & 0\\
			0 & U_1
		\end{bmatrix},\dots,\begin{bmatrix}
			M_{\Phi_{n-1}} & 0\\
			0 & U_{n-1}
		\end{bmatrix},\begin{bmatrix}
			M_z & 0\\
			0 & U_n
		\end{bmatrix}\right).
	\end{equation*}Next, we show that the operators $E_1,\dots,E_{n-1}$ obtained from equation \eqref{eq2.9} satisfy the commutativity condition \eqref{eq1.1}. Applying equation \eqref{eq2.6} it yields that $$\begin{bmatrix}
		M_{\Phi_i}	M_{\Phi_j} & 0\\
		0 & U_iU_j
	\end{bmatrix}=\begin{bmatrix}
		M_{\Phi_j}M_{\Phi_i} & 0\\
		0 & U_jU_i
	\end{bmatrix},$$ for all $i,j = 1, \dots, n-1$. Observe that the commutativity of the multiplication operator $M_{\Phi_i}$, that is,
	\begin{equation}\label{eq2.10}
		M_{\Phi_i}M_{\Phi_j}=M_{\Phi_j}M_{\Phi_i}, 
	\end{equation} for all $i,j = 1, \dots, n-1$, implies that the operators $E_1,\dots,E_{n-1}$ obtained from equation \eqref{eq2.9} satisfy the commutativity condition \eqref{eq1.1}.
	%	Next, from the assumption that $(\gamma_1 M_{\Phi_1},\ldots,\gamma_{n-1} M_{\Phi_{n-1}})$ is $\Gamma_{n-1}$-contraction, equation \eqref{eq2.10}, equation \eqref{eq2.9} and Theorem \ref{thm1.4} we obtain that the tuple $(M_{\Phi_1},\dots,M_{\Phi_{n-1}},M_z)$ is $\Gamma_n$-isometry. 
	Finally, we show that $(U_1,\dots,U_n)$ is a $\Gamma_n$-unitary. Since $(V_1,\ldots,V_n)$ is assumed to be a $\Gamma_n$-isometry, an application of Lemma~\ref{lem2.1}, parts $(1)$ and then $(2)$, yields that the tuple $(U_1,\dots,U_n)$ is a $\Gamma_n$-contraction. Furthermore, as $U_n$ is unitary, it follows that $(U_1,\dots,U_n)$ is in fact a $\Gamma_n$-unitary. We now proceed to establish the converse implication. Suppose that \begin{equation*}
		(V_1,\ldots,V_{n-1},V_n)=U^*\left(\begin{bmatrix}
			M_{\Phi_1} & 0\\
			0 & U_1
		\end{bmatrix},\dots,\begin{bmatrix}
			M_{\Phi_{n-1}} & 0\\
			0 & U_{n-1}
		\end{bmatrix},\begin{bmatrix}
			M_z & 0\\
			0 & U_n
		\end{bmatrix}\right)U,
	\end{equation*} for some unitary operator $U$. Our goal is to establish that $(V_1,\ldots,V_n)$ is a $\Gamma_n$-isometry. It is clear that $V_n$ is an isometry. From equation \eqref{eq2.9}, \eqref{eq2.10}, together with the assumption that $(U_1,\dots,U_n)$ is a $\Gamma_n$-unitary, it follows that the operators $V_1,\dots,V_n$ are mutually commuting and satisfy the relation $V_i=V_{n-i}^*V_n$ for each $i=1,\dots,n-1$. Since the tuple $(\gamma_1V_1,\dots,\gamma_{n-1}V_{n-1})$ is a $\Gamma_{n-1}$-contraction, therefore, by Theorem \ref{thm1.2} we conclude that $(V_1,\ldots,V_n)$ is a $\Gamma_n$-isometry. This completes the proof.
	%the block diagonal tuple
	%	\begin{equation*}
		%		\left(\begin{bmatrix}
			%			M_{\Phi_1} & 0\\
			%			0 & T_1
			%		\end{bmatrix},\dots,\begin{bmatrix}
			%			M_{\Phi_{n-1}} & 0\\
			%			0 & T_{n-1}
			%		\end{bmatrix},\begin{bmatrix}
			%			M_z & 0\\
			%			0 & T_n
			%		\end{bmatrix}\right)
		%	\end{equation*}
	%	is a $\Gamma_n$-isometry. Finally, applying Lemma \ref{lem2.1} part $(1)$
	
\end{proof}
%The preceding Theorem establishes that statement $(1)\implies (2)$. With the condition in $(2)$, if one assumes that $(\gamma_1 M_{\Phi_1},\ldots,\gamma_{n-1} M_{\Phi_{n-1}})$ is a $\Gamma_{n-1}$-contraction, then $(1)\iff(2)$.
The following corollary demonstrates that, under suitable hypothesis, $\Gamma_n$-isometry extends to a $\Gamma_{n}$-unitary.
\begin{cor}\label{cor2.4}
	Assume that $\Phi_1,\dots,\Phi_{n-1}\in H^{\infty}(\mathcal{B}(\mathcal{E}))$ are obtained as in Theorem \ref{thm2.3} such that $(\gamma_1\Phi_1(z),\dots,\gamma_{n-1}\Phi_{n-1}(z))$ is a $\Gamma_{n-1}$-contraction for all $z\in \mathbb{T}$. If $(V_1,\dots,V_n)$ is a $\Gamma_n$-isometry then it extends to a $\Gamma_n$-unitary acting on the space of the minimal unitary extension of the isometry $S_n$.
	%	 $(\gamma_1 M_{\Phi_1},\ldots,\gamma_{n-1} M_{\Phi_{n-1}})$ acting on $H^2(\mathcal{E})$ is a $\Gamma_{n-1}$-contraction and  
\end{cor}
\begin{proof}
	Suppose $(V_1,\dots,V_n)$ is a $\Gamma_n$-isometry. By virtue of Theorem~\ref{thm2.3}, we can assume without loss of generality that $(V_1,\dots,V_n)$ admits the following form
	\begin{equation*}
		\left(\begin{bmatrix}
			M_{\Phi_1(z)} & 0\\
			0 & U_1
		\end{bmatrix},\dots,\begin{bmatrix}
			M_{\Phi_{n-1}(z)} & 0\\
			0 & U_{n-1}
		\end{bmatrix},\begin{bmatrix}
			M_z & 0\\
			0 & U_n
		\end{bmatrix}\right):\begin{bmatrix}
			H_{\mathcal{E}}^2(\mathbb{D})\\
			\mathcal{F}
		\end{bmatrix} \rightarrow \begin{bmatrix}
			H_{\mathcal{E}}^2(\mathbb{D})\\
			\mathcal{F}
		\end{bmatrix}
	\end{equation*}
	for some Hilbert spaces $\mathcal{E}$ and $\mathcal{F}$, operators $E_1,\dots,E_{n-1}$ on $\mathcal{E}$ satisfying the commutativity condition \eqref{eq1.1} and a $\Gamma_n$-unitary $(U_1,\dots,U_n)$ acting on $\mathcal{F}$. Here $M_{\Phi_i}$ denotes the multiplication operator on $H_{\mathcal{E}}^2(\mathbb{D})$. We regard $H_{\mathcal{E}}^2(\mathbb{D})\oplus \mathcal{F}$ as a subspace of $L_{\mathcal{E}}^2(\mathbb{T})\oplus \mathcal{F}$ in the natural way. With this identification, for $\eta \in \mathbb{T}$, the operator tuple  
	\begin{equation*}
		\left(\begin{bmatrix}
			M_{\Phi_1(\eta)} & 0\\
			0 & U_1
		\end{bmatrix},\dots,\begin{bmatrix}
			M_{\Phi_{n-1}(\eta)} & 0\\
			0 & U_{n-1}
		\end{bmatrix},\begin{bmatrix}
			M_{\eta} & 0\\
			0 & U_n
		\end{bmatrix}\right):\begin{bmatrix}
			L_{\mathcal{E}}^2(\mathbb{T})\\
			\mathcal{F}
		\end{bmatrix} \rightarrow \begin{bmatrix}
			L_{\mathcal{E}}^2(\mathbb{T})\\
			\mathcal{F}
		\end{bmatrix}
	\end{equation*}
	serves as the natural extension of the tuple $(V_1,\dots,V_n)$ defined on $ H_{\mathcal{E}}^2(\mathbb{D}) \oplus \mathcal{F}$. Now, we prove that the following operator tuple
	\begin{equation*}
		\left(\begin{bmatrix}
			M_{\Phi_1(\eta)} & 0\\
			0 & U_1
		\end{bmatrix},\dots,\begin{bmatrix}
			M_{\Phi_{n-1}(\eta)} & 0\\
			0 & U_{n-1}
		\end{bmatrix},\begin{bmatrix}
			M_{\eta} & 0\\
			0 & U_n
		\end{bmatrix}\right)
	\end{equation*}
	is a $\Gamma_n$-unitary. It is easy to observe that the operator $\begin{bmatrix}
		M_{\eta} & 0\\
		0 & U_n
	\end{bmatrix}$ is unitary on $L_{\mathcal{E}}^2(\mathbb{T}) \oplus \mathcal{F}$, and it constitutes the minimal unitary extension of the isometry
	$\begin{bmatrix}
		M_z & 0 \\
		0 & U_n
	\end{bmatrix}
	$
	acting on $H_{\mathcal{E}}^2(\mathbb{D}) \oplus \mathcal{F}$. 
	Using the characterization of $\Gamma_n$-isometry, we know that $$\begin{bmatrix}
		M_{\Phi_i(z)} & 0\\
		0 & U_i
	\end{bmatrix}\begin{bmatrix}
		M_{\Phi_j(z)} & 0\\
		0 & U_j
	\end{bmatrix}=\begin{bmatrix}
		M_{\Phi_j(z)} & 0\\
		0 & U_j
	\end{bmatrix}\begin{bmatrix}
		M_{\Phi_i(z)} & 0\\
		0 & U_i
	\end{bmatrix}.$$
	for all $i,j=1,\dots,n-1$. This operator identity implies that $U_i$ and $U_j$ commute, that is, $U_iU_j=U_jU_i$ and the Toeplitz operator symbols equal to the operator pencil $\Phi_i(z)$ and $\Phi_j(z)$ also commute
	\begin{equation*}
		\Phi_i(z)\Phi_j(z)=\Phi_j(z)\Phi_i(z)
	\end{equation*}  
	for all $i,j=1,\dots,n-1$. As a consequence, the associated Laurent multiplication operators acting on $L_{\mathcal{E}}^2(\mathbb{T})$ commutes  
	\begin{equation}\label{ch1}
		M_{\Phi_i(\eta)}M_{\Phi_j(\eta)}=M_{\Phi_j(\eta)}M_{\Phi_i(\eta)}.
	\end{equation}
	for all $i,j=1,\dots,n-1$. Moreover
	\begin{equation}\label{ch2}
		M_{\Phi_i(\eta)}M_{\eta}=M_{\eta}M_{\Phi_i(\eta)},
	\end{equation}
	and
	\begin{equation}\label{ch3}
		M_{\Phi_i(\eta)}=M^*_{\Phi_{n-i}(\eta)}M_{\eta}
	\end{equation}
	for all $i=1,\dots,n-1$. By \cite[Lemma 4.8]{sbs}, $(\gamma_1M_{\Phi_1(\eta)},\dots,\gamma_{n-1}M_{\Phi_{n-1}(\eta)})$ is a $\Gamma_{n-1}$-contraction if and only if the tuple $(\gamma_1\Phi_1(\eta),\dots,\gamma_{n-1}\Phi_{n-1}(\eta))$ is a $\Gamma_{n-1}$-contraction for all $\eta\in \mathbb{T}$. This fact along with the condition in equation \eqref{ch1},\eqref{ch2}, and \eqref{ch3} implies that  $(M_{\Phi_1(\eta)},\dots,M_{\Phi_{n-1}(\eta)},M_{\eta})$ is a $\Gamma_n$-unitary. Furthermore, since $(U_1,\dots,U_n)$ and $(M_{\Phi_1(\eta)},\dots,M_{\Phi_{n-1}(\eta)},M_{\eta})$ are $\Gamma_n$-unitaries, it follows from Lemma \ref{lem2.2} that the block diagonal operator tuple \begin{equation*}
		\left(\begin{bmatrix}
			M_{\Phi_1(\eta)} & 0\\
			0 & U_1
		\end{bmatrix},\dots,\begin{bmatrix}
			M_{\Phi_{n-1}(\eta)} & 0\\
			0 & U_{n-1}
		\end{bmatrix},\begin{bmatrix}
			M_{\eta} & 0\\
			0 & U_n
		\end{bmatrix}\right),
	\end{equation*}
	is also a $\Gamma_n$-unitary. This completes the proof.
	%Conversely, if a tuple extends to a $\Gamma_n$-unitary then by Theorem \ref{thm1.4} it is also a $\Gamma_n$-isometry. 
\end{proof}
Next, we use Corollary \ref{cor2.4} to construct the canonical $\Gamma_n$-unitary from a $\Gamma_n$-contraction. Let $(S_1, \dots, S_n)$ be a $\Gamma_n$-contraction acting on a Hilbert space $\mathcal{H}$. It is well known that for a $\Gamma_n$-contraction $(S_1, \dots, S_n)$, the following operator norm bounds hold (see \cite[Proposition 4.1]{19})
\begin{equation*}
	\|S_i\| \leq \binom{n-1}{i} + \binom{n-1}{n-i}, \quad \text{for } i = 1, \dots, n-1, \quad \text{and} \quad \|S_n\| \leq 1.
\end{equation*}
Moreover, the contraction $S_n$ satisfies the following chain of operator inequalities
\begin{equation*}
	I \geq S_n S_n^* \geq S_n^2 {S_n^*}^2 \geq \cdots \geq S_n^k {S_n^*}^k \geq \cdots \geq 0.
\end{equation*}
Hence, there exist a positive semi definite operator $\mathcal P_{S_n^*}$ such that 
\begin{equation*}
	{\mathcal P_{S_n^*}}^2=\text{SOT}-\lim S_n^k{S_n^*}^k.
\end{equation*} Define the operators $X^*_{S_n^*},{S_i^*}_{S_n^*}:\overline{\text{Ran}}\mathcal P_{S_n^*}\rightarrow\overline{\text{Ran}}\mathcal P_{S_n^*}$ densely by \begin{equation}
	X^*_{S_n^*}\mathcal P_{S_n^*}h=\mathcal P_{S_n^*}S_n^*h,
\end{equation}
and
\begin{equation}\label{eq-def}
	{S_i^*}_{S_n^*}\mathcal P_{S_n^*}h=
	\mathcal P_{S_n^*}S_i^*h,
\end{equation}
for all $i=1,\dots,n-1$.
%Using Equation \eqref{eq2.2} it is easy to verify that $X^*_{S_n^*}$ is an isometry. Moreover, for all $i=0,\dots n-1$
%\begin{align*}
%	\langle S_i{\mathcal P_{S_n^*}}^2S_i^*h,h\rangle&=\langle\lim_{k\rightarrow \infty}S_n^kS_iS_i^*{S_n^*}^kh,h\rangle\\
%	&\leq \binom{n-1}{i}+ \binom{n-1}{n-i}\lim_{k\rightarrow \infty}\langle S_nS_n^*h,h\rangle\\ 
%	&=\binom{n-1}{i}+ \binom{n-1}{n-i}\langle{\mathcal P_{S_n^*}}^2h,h\rangle.
%\end{align*} 
%As a consequence, the operators ${S_i^*}_{S_n^*}:\overline{\text{Ran}}\mathcal P_{S_n^*}\rightarrow\overline{\text{Ran}}\mathcal P_{S_n^*}$ defined densely by 
% also obeys the condition $\|{S_i^*}_{S_n^*}\|\leq\binom{n-1}{i}+ \binom{n-1}{n-i}$. Hence using (\textcolor{red}{Chapter X, 317}) and \textcolor{red}{Hahn Banach Theorem} ${S_i^*}_{S_n^*}$ can be extended to all of $\overline{\text{Ran}}\mathcal P_{S_n^*}.$ 
Observe that the commutativity of $(S_1^*,\dots,S_n^*)$ guarantees that the tuple $({S_1^*}_{S^*_n},\dots,{S_{n-1}^*}_{S^*_n},X^*_{S_n^*})$ is also commutative. Furthermore, the tuple $({S_1^*}_{S^*_n},\dots,{S_{n-1}^*}_{S^*_n},X^*_{S_n^*})$ is a $\Gamma_n$-contraction. Infact, if $F$ is a $n$-variable polynomial then for all $h\in \mathcal{H}$ we have  
\begin{align*}
	\|F({S_1^*}_{S^*_n},\dots,{S_{n-1}^*}_{S^*_n},X^*_{S_n^*})\mathcal P_{S_n^*}h\|&=\|\mathcal P_{S_n^*}F(S_1^*,\dots,S_n^*)h\|\\
	&\leq\|F(S_1^*,\dots,S_n^*)\|\|h\|\\
	&\leq(\sup_{\Gamma_n}|f|)\|h\|.
\end{align*}
Since $X^*_{S_n^*}$ is an isometry, Theorem \ref{thm1.4} ensures that  $({S_1^*}_{S^*_n},\dots,{S_{n-1}^*}_{S^*_n},X^*_{S_n^*})$ is a $\Gamma_n$-isometry. Now, assume that $\hat{\Phi}_1,\dots,\hat{\Phi}_{n-1}\in H^{\infty}(\mathcal{B}(\mathcal{E}))$ are obtained from the $\Gamma_n$-isometry $({S_1^*}_{S^*_n},\dots,{S_{n-1}^*}_{S^*_n},X^*_{S_n^*})$ (existence follows from Theorem \ref{thm2.3}) such that, for every $z\in \mathbb{T}$, the tuple $(\gamma_1\hat{\Phi}_1(z),\dots,\gamma_{n-1}\hat{\Phi}_{n-1}(z))$ acting on $H^{\infty}(\mathcal{B}(\mathcal{E}))$ is a $\Gamma_{n-1}$-contraction. Finally, applying Corollary \ref{cor2.4}, we obtain that the tuple  $({S_1^*}_{S^*_n},\dots,{S_{n-1}^*}_{S^*_n},X^*_{S_n^*})$ has a unitary extension $(D_1^*,\dots,D_n^*)$ acting on the space $\mathscr{P}_{S_n^*}\supseteq\overline{\text{Ran}}\mathcal P_{S_n^*}$, where $D_n^*$ acting on $\mathscr{P}_{S_n^*}$ is the minimal unitary extension for $X^*_{S_n^*}$.

\begin{defn}
	Consider a $\Gamma_n$-contraction $(S_1,\dots,S_n)$  and assume that $(D_1,\dots,D_n)$ is the $\Gamma_n$-unitary constructed from it as described earlier. This operator tuple $(D_1,\dots,D_n)$ is referred to as the \textit{canonical $\Gamma_n$-unitary} associated with the $\Gamma_n$-contraction $(S_1,\dots,S_n)$.
\end{defn}

The next theorem shows that the canonical $\Gamma_n$-unitary corresponding to a given $\Gamma_n$-contraction $(S_1,\dots,S_n)$ is unique up to unitary equivalence.
\begin{thm}
	Let $(D_1,\dots,D_n)$ and $(D_1^{\prime},\dots,D_n^{\prime})$ be the canonical $\Gamma_n$-unitaries associated with the $\Gamma_n$-contractions $(S_1,\dots,S_n)$ on the Hilbert space $\mathcal{H}$  and $(S_1^{\prime},\dots,S_n^{\prime})$ on the Hilbert space  $\mathcal{H}^{\prime}$, respectively. Suppose $(S_1,\dots,S_n)$ and $(S_1^{\prime},\dots,S_n^{\prime})$ are unitary equivalent through $U$ then $(D_1,\dots,D_n)$ and $(D_1^{\prime},\dots,D_n^{\prime})$ are unitary equivalent through $U_D:\mathscr{P}_{S_n^*}\rightarrow\mathscr{P}_{S_n^{{\prime*}}}$ given by 
	\begin{equation}\label{eq2.5}
		U_D:D^k_n\mathcal{P}_{S_n^*}h\rightarrow {D^{\prime}}^k_n\mathcal{P}_{S_n^{\prime*}}Uh
	\end{equation}
	for all $k\geq0$ and $h\in \mathcal H$.
\end{thm}
\begin{proof}
	Assume that $D_n^*,~{D_n^*}^{\prime}$ acting on $\mathscr{P}_{S_n^*},~ \mathscr{P}_{{S_n^{\prime}}^*}$, and the set of operators $\{{S_1}_{S^*_n},\dots,{S_{n-1}}_{S^*_n},\mathcal P_{S_n^*}\}$, $\{{S_1}_{{S_n^{\prime}}^*}, \dots,{S_{n-1}}_{{S_n^{\prime}}^*},\mathcal P_{{S_n^{\prime}}^*}\}$ are obtained from $(S_1,\dots,S_n)$, and $(S_1^{\prime},\dots,S_n^{\prime})$, respectively. 
	Since the tuples $(S_1,\dots,S_n)$ and $(S_1^{\prime},\dots,S_n^{\prime})$ 
	are unitarily equivalent through a unitary operator $U$, 
	it follows that $U$ intertwines $S_n$ with $S_n^{\prime}$, 
	and hence also intertwines their adjoints, that is, 
	$US_n^* = S_n^{\prime *}U.$
	In addition, $U$ intertwines $\mathcal{P}_{S_n^*}$ and 
	$\mathcal{P}_{S_n^{\prime *}}$, that is,
	$$U \mathcal{P}_{S_n^*} = \mathcal{P}_{S_n^{\prime *}} U.$$
	Consequently, $U$ has the following intertwining property, 
	\begin{equation}\label{inter-can}
		U({S_1}_{S_n^*},\dots,{S_{n-1}}_{S_n^*},\mathcal{P}_{S_n^*})=({S_1}_{{S_n^{\prime}}^*}, \dots,{S_{n-1}}_{{S_n^{\prime}}^*},\mathcal P_{{S_n^{\prime}}^*})U.
	\end{equation}
	From equation~\eqref{eq2.5}, it then follows that the operator $U_D$ intertwines 
	$D_n$ with $D_n^{\prime}$. To conclude that $(D_1,\dots,D_n)$ and $(D_1^{\prime},\dots,D_n^{\prime})$ are unitarily equivalent via $U_D$, it remains only to verify that for every $h \in \mathcal{H}$ and every $k \geq 0$.
	\begin{equation*}
		U_DD_iD_n^k\mathcal{P}_{S_n^*}h=D^{\prime}_iU_DD_n^k\mathcal{P}_{S_n^*}h, 
	\end{equation*} 
	for all $i=1,\dots,n-1$. Since $(D_1,\dots,D_n)$ is a $\Gamma_n$-unitary, and that $U_D$ intertwines 
	$D_n$ with $D_n^{\prime}$, it follows that for every $h\in \mathcal{H}$ and every $k\geq0$, we have
	\begin{align}\label{equality-1}
		U_DD_iD_n^k\mathcal{P}_{S_n^*}h&\nonumber=U_DD_n^kD_i\mathcal{P}_{S_n^*}h\\
		&\nonumber=U_DD_n^{k+1}D_n^*D_i\mathcal{P}_{S_n^*}h\\
		&\nonumber=U_DD_n^{k+1}D_{n-i}^*\mathcal{P}_{S_n^*}h\\
		&={D_n^{\prime}}^{k+1}U_DD_{n-i}^*\mathcal{P}_{S_n^*}h.
	\end{align}
	Note that the $\Gamma_n$-isometry $({S_1^*}_{S^*_n},\dots,{S_{n-1}^*}_{S^*_n},X^*_{S_n^*})$ acting on $\operatorname{Ran}\mathcal{P}_{S_n^*}$ admits a $\Gamma_n$-unitary extension $(D_1^*,\dots,D_n^*)$ acting on $\mathscr{P}_{S_n^*}\supseteq \operatorname{Ran}\mathcal{P}_{S_n^*}$. Therefore, by first applying \eqref{eq2.5} and then \eqref{inter-can}, we obtain 
	\begin{equation}\label{equality-2}
		{D_n^{\prime}}^{k+1}U_DD_{n-i}^*\mathcal{P}_{S_n^*}h=	{D_n^{\prime}}^{k+1}U{S_{n-i}^*}_{S_n^*}\mathcal{P}_{S_n^*}h
		={D_n^{\prime}}^{k+1}{S_{n-i}^*}_{{S_n^{\prime}}^*}\mathcal{P}_{{S_n^{\prime}}^*}Uh.		
	\end{equation}
	Moreover, the $\Gamma_n$-isometry $({S_1^*}_{S^{\prime*}_n},\dots,{S_{n-1}^*}_{S^{\prime*}_n},X^*_{S_n{\prime*}})$ acting on $\operatorname{Ran}\mathcal{P}_{S_n^{\prime*}}$ has a $\Gamma_n$-unitary extension $(D_1^{\prime*},\dots,D_n^{\prime*})$ acting on $\mathscr{P}_{S_n^{\prime*}}\supseteq \operatorname{Ran}\mathcal{P}_{S_n^{\prime*}}$. Hence, 
	\begin{equation}\label{equality-3}
	{D_n^{\prime}}^{k+1}{S_{n-i}^*}_{{S_n^{\prime}}^*}\mathcal{P}_{{S_n^{\prime}}^*}Uh	={D_n^{\prime}}^{k+1}{D_{n-i}^{\prime*}}\mathcal{P}_{{S_n^{\prime}}^*}Uh.
	\end{equation}
	Finally, using the fact that $(D_1^{\prime},\dots,D_n^{\prime})$ is the $\Gamma_n$-unitary and applying \eqref{eq2.5}, we get
	\begin{align}\label{equality-4}
		{D_n^{\prime}}^{k+1}{D_{n-i}^{\prime*}}\mathcal{P}_{{S_n^{\prime}}^*}Uh & \nonumber={{D^{\prime}}^*_{n-i}}{D_n^{\prime}}^{k+1}\mathcal{P}_{{S_n^{\prime}}^*}Uh\\
		&\nonumber=D^{\prime}_i{D_n^{\prime}}^k\mathcal{P}_{{S_n^{\prime}}^*}Uh\\
		&=D^{\prime}_iU_DD_n^k\mathcal{P}_{S_n^*}h
	\end{align}
	The identities in equation \eqref{equality-1}, \eqref{equality-2}, \eqref{equality-3}, and \eqref{equality-4} together yield $$U_DD_iD_n^k\mathcal{P}_{S_n^*}h=D^{\prime}_iU_DD_n^k\mathcal{P}_{S_n^*}h.$$ This completes the proof.
\end{proof}
	
\section{Douglas-Type Functional Model}
In this section, we develop a Douglas-type functional model for $\Gamma_n$-contractions. The construction relies on the fundamental operator tuple $(E_1, \dots, E_{n-1})$ associated with the adjoint tuple $(S_1^*, \dots, S_n^*)$.
%For $i =1,\dots,n-1$. Note that, using Theorem \ref{thm1.2} and Theorem \ref{thm1.3}, we obtain that there exist exactly one solution $(X_1,\dots,X_{n-1})=(E_1,\dots,E_{n-1})$ to the system of equations:
%\begin{equation}
%	S_i^*-S_{n-i}S_n^*=D_{S_n^*}E_iD_{S_n^*}~\text{and}~S_{n-i}^*-S_i S_n^*=D_{S_n^*}E_{n-i}D_{S_n^*},
%\end{equation}
%Equivalently, the fundamental operators $(E_1, \dots, E_{n-1})$ provide the unique solution to the alternative systems:
%\begin{equation}\label{eq3.20}
%	D_{S_n^*} S_i^* = X_i D_{S_n^*}+ X^*_{n-i} D_{S_n^*} S_n^*.
%\end{equation}
The proposition below establishes the necessary and sufficient conditions on the fundamental operators $E_1, \dots, E_{n-1}$ for which the tuple $(M_{\Phi_1}, \dots, M_{\Phi_{n-1}}, M_z)$ acting on the Hilbert space $H_{\mathcal{D}_{S_n^*}}^2(\mathbb{D})$, constitutes a pure $\Gamma_n$-isometry, where $\Phi_i(z)=E_i^*+E_{n-i}z$, for $z\in \mathbb{D}$. 

\begin{prop}\label{prop3.1}
	Let $(E_1, \dots, E_{n-1})$ be the fundamental operator tuple associated with the $\Gamma_n$-contraction $(S_1^*, \dots, S_n^*)$. The tuple $(M_{\Phi_1}, \dots, M_{\Phi_{n-1}},$ $M_z)$ acting on the Hilbert space $H_{\mathcal{D}_{S_n^*}}^2(\mathbb{D})$ is a pure $\Gamma_n$-isometry if and only if the tuple $(\gamma_1 M_{\Phi_1}, \dots, \gamma_{n-1} M_{\Phi_{n-1}})$ is a $\Gamma_{n-1}$-contraction and the fundamental operators $E_1, \dots, E_{n-1}$ satisfy the commutativity condition given in Equation~\eqref{eq1.1} and .
\end{prop}
\begin{proof}
	
	Assume that the operator tuple $(M_{\Phi_1}, \dots, M_{\Phi_{n-1}}, M_z)$ acting on the Hilbert space $H_{\mathcal{D}_{S_n^*}}^2(\mathbb{D})$ is a $\Gamma_n$-isometry. Then, by Theorem~\ref{thm1.4},  $(M_{\Phi_1},\dots,M_{\Phi_{n-1}},M_z)$ is commutative and $(\gamma_1 M_{\Phi_1}, \dots, \gamma_{n-1} M_{\Phi_{n-1}})$ is a $\Gamma_{n-1}$-contraction. As a result, the fundamental operators $(E_1, \dots, E_{n-1})$ associated with the adjoint tuple $(S_1^*, \dots, S_n^*)$ must satisfy the commutativity condition in equation~\eqref{eq1.1}. Conversely, suppose that the fundamental operators $(E_1, \dots, E_{n-1})$ satisfy the commutativity condition~\eqref{eq1.1}. Then the operators $M_{\Phi_1}, \dots, M_{\Phi_{n-1}}, M_z$ are mutually commuting,  $M_{\Phi_i}=M_{\Phi_{n-i}}^*M_z$. Consequently, the assumption that $(\gamma_1 M_{\Phi_1}, \dots, \gamma_{n-1} M_{\Phi_{n-1}})$ is a $\Gamma_{n-1}$- contraction implies that $(M_{\Phi_1},\dots,M_{\Phi_{n-1}},M_z)$ on $H_{\mathcal{D}_{S_n^*}}^2(\mathbb{D})$ is a $\Gamma_n$-isometry.
\end{proof}
Now, we develop the Douglas type functional model for $\Gamma_n$-contraction. Since $S_n$ is a contraction, for each $h\in \mathcal{H}$ we define $\mathscr{D}_{D_{S_n^*},S_n^*}:\mathcal{H}\rightarrow H_{\mathcal{D}_{S_n^*}}^2(\mathbb{D})$ by
\begin{equation}\label{eq3.24}
	\mathscr{D}_{D_{S_n^*},S_n^*}(z)h:=\sum_{k=0}^{\infty}z^nD_{S_n^*}{S_n^*}^kh=D_{S_n^*}(I-zS_n^*)^{-1}h,
\end{equation}
and $\mathscr{E}_D:\mathcal H\rightarrow H^2(\mathcal{D}_{S_n^*})\oplus
\mathscr{P}_{S_n^*}$ by 
\begin{equation}\label{eq3.25}
	\mathscr{E}_Dh=\begin{bmatrix}
		\mathscr{D}_{D_{S_n^*},S_n^*}(z)h\\
		\mathcal{P}_{S_n^*}h
	\end{bmatrix}~\text{for all}~h\in \mathcal{H}.
\end{equation}
The map $\mathscr{E}_D$ is an isometry. Moreover, by the intertwining property of $\mathscr{D}_{D_{S_n^*},S_n^*}$ and $\mathcal{P}_{S_n^*}$, it follows that 
\begin{align}\label{eq3.26}
	\mathscr{E}_DS_n^*=\begin{bmatrix}
		\mathscr{D}_{D_{S_n^*},S_n^*}S_n^*\\ 
		\mathcal{P}_{S_n^*}S_n^*
	\end{bmatrix}&=\begin{bmatrix}
		M_z^*\mathscr{D}_{D_{S_n^*},S_n^*}\\ \notag
		\mathcal{P}_{S_n^*}S_n^*
	\end{bmatrix}\\ &=\begin{bmatrix}
		M_z^*\mathscr{D}_{D_{S_n^*},S_n^*}\\ \notag
		X^*_{S_n^*}\mathcal{P}_{S_n^*}
	\end{bmatrix}\\ &=\begin{bmatrix}
		M_z^*\mathscr{D}_{D_{S_n^*},S_n^*}\\
		D_n^*\mathcal{P}_{S_n^*}
	\end{bmatrix}=\begin{bmatrix}
		M_z & 0\\
		0 & D_n
	\end{bmatrix}^*\mathscr{E}_D.
\end{align}
Hence, $\begin{bmatrix}
	M_z & 0\\
	0 & D_n
\end{bmatrix}$ is an isometric dilation of $S_n$ corresponding to the isometric embedding operator $\mathscr{E}_D$. This construction originates from the work of Douglas (see \cite{10}), where it was established that the dilation is minimal. In the subsequent theorem, the lift is expressed in a functional form aligned with the coordinate system of the Douglas model. 
%As discussed earlier, under suitable assumption a $\Gamma_n$-isometry extends to a $\Gamma_n$-unitary(see Corollary \ref{cor2.4}). Moreover, in view of Proposition \ref{prop3.1}, we shall adopt the following assumption throughout the remainder of  paper:
%\newline \textit{Assumption}: The functions $\Phi_1,\dots,\Phi_{n-1}\in H^{\infty}(\mathcal{B}(\mathcal{D}_{S_n^*}))$ are the bounded analytic functions obtained in Theorem \ref{thm2.3} such that $(\gamma_1 M_{\Phi_1}, \dots, \gamma_{n-1} M_{\Phi_{n-1}})$ is a $\Gamma_{n-1}$-contraction.

\begin{thm}\label{thm3.2}
	\textit{(Douglas-Type Functional Model)} Let $(S_1,\dots,S_n)$ be a $\Gamma_n$-contraction on a Hilbert space, and $(E_1,\dots,E_{n-1})$ be the fundamental operators of $(S_1^*,\dots,S_n^*)$ on $\mathcal{H}$ that satisfies the commutativity condition \eqref{eq1.1}. Let $\Phi_1,\dots,\Phi_{n-1}\in H^{\infty}(\mathcal{B}(\mathcal{D}_{S_n^*}))$ be obtained as in Theorem \ref{thm2.3} such that $(\gamma_1 M_{\Phi_1}, \dots, \gamma_{n-1} M_{\Phi_{n-1}})$ is a $\Gamma_{n-1}$-contraction. If $(D_1,\dots,D_n)$ be the $\Gamma_n$-unitary canonically constructed from $(S_1,\dots,S_n)$. Then $(S_1,\dots,S_n)$ is jointly equivalent to $P_{\mathcal{H}_{D}}(\hat{V_1},\dots,\hat{V_n})|_{\mathcal{H}_{D}}$, where $(\hat{V_1},\dots,\hat{V_n})$ is equal to  
	\begin{equation*}
		\left( \begin{bmatrix}
			M_{\Phi_1} & 0\\
			0 & D_1
		\end{bmatrix},\dots,\begin{bmatrix}
			M_{\Phi_{n-1}} & 0\\
			0 & D_{n-1}
		\end{bmatrix},\begin{bmatrix}
			M_z & 0\\
			0 & D_n 
		\end{bmatrix} \right),
	\end{equation*} 
	and $\mathcal{H}_{D}$ is the functional model space given by $\mathcal{H}_{D}=\operatorname{Ran}\mathscr{E}_D$. For this particular scenario, the operator tuple
	\begin{equation}\label{eq-modeld}
		\left( 
		\begin{bmatrix}
			M_{\Phi_1} & 0 \\
			0 & D_1
		\end{bmatrix},
		\dots,
		\begin{bmatrix}
			M_{\Phi_{n-1}} & 0 \\
			0 & D_{n-1}
		\end{bmatrix}
		\right)
	\end{equation}
	acting on the Hilbert space $ H_{\mathcal{D}_{S_n^*}}^2(\mathbb{D})\oplus
	\mathscr{P}_{S_n^*}$ 
	is uniquely determined once a minimal isometric dilation of 
	$S_n$ is fixed, given by the block operator
	$$\begin{bmatrix}
		M_z & 0 \\
		0 & D_n
	\end{bmatrix}.$$
\end{thm}
\begin{proof}
	%		Let $(S_1,\dots,S_n)$ be a $\Gamma_n$-contraction then $S_n$ is a contraction on $\mathcal H$. Hence, using Douglas Model for the contraction $S_n$, there is an isometric embedding operator $\mathscr{E}$ that intertwines $S_n^*$ and $\begin{bmatrix}
		%			M_z & 0\\
		%			0 & D
		%		\end{bmatrix}^*$, where $\mathscr{E}:\mathcal H\rightarrow 
	%		\begin{bmatrix}
		%		H^2(\mathcal{D}_{S_n^*})\\
		%		\mathscr{P}_{S_n^*}
		%		\end{bmatrix}$ given by \begin{equation}\label{eq3.17}
		%		\mathscr{E}h=\begin{bmatrix}
			%			\mathscr{D}_{D_{S_n^*},S_n^*}(z)h\\
			%			\mathcal{P}_{S_n^*}h
			%		\end{bmatrix}~\text{for all}~h\in \mathcal{H}.
		%		\end{equation}
	%		Now, 
	Consider $(S_1,\dots,S_n)$ is a $\Gamma_n$-contraction on the Hilbert space $\mathcal{H}$. Let $(E_1,\dots,E_{n-1})$ be the fundamental operator tuple associated with the adjoint tuple $(S_1^*,\dots,S_n^*)$, and let $(D_1,\dots, D_n)$, acting on $\mathscr{P}_{S_n^*}$ be the canonical $\Gamma_n$-unitary constructed from the $\Gamma_n$-contraction $(S_1,\dots,S_n)$. We claim that the isometric embedding operator  
	$\mathscr{E}_D$ defined in equation \eqref{eq3.25} will intertwine the operator tuple $(S_1^*,\dots,S_n^*)$ and $(\hat{V_1}^*,\dots,\hat{V_n}^*)$, that is, 
	\begin{equation}\label{eq3.28}
		\mathscr{E}_D(S_1^*,\dots,S_n^*)=(\hat{V_1}^*,\dots,\hat{V_n}^*)\mathscr{E}_D.
	\end{equation} 
	To establish the intertwining relation stated in \eqref{eq3.28}, we verify the following intertwining relations
	\begin{equation*}
		\mathscr{D}_{D_{S_n^*},S_n^*}(S_1^*,\dots,S_n^*)=(M_{\Phi_1}^*,\dots,M_{\Phi_{n-1}}^*,M_z^*)\mathscr{D}_{D_{S_n^*},S_n^*},
	\end{equation*}
	and
	\begin{equation*}
		\mathcal{P}_{S_n^*}(S_1^*,\dots,S_n^*)=(D_1^*,\dots,D_n^*)\mathcal{P}_{S_n^*}. 
	\end{equation*}
	Note that, the equalities $\mathscr{D}_{D_{S_n^*},S_n^*}S_n^*=M_z^*\mathscr{D}_{D_{S_n^*},S_n^*}$ and $\mathcal{P}_{S_n^*}S_n^*=D_n^*\mathcal{P}_{S_n^*}$ follows from the fact that $\begin{bmatrix}
		M_z & 0\\
		0 & D_n
	\end{bmatrix}$ is the Douglas minimal isometric dilation of $S_n$ with the isometric embedding operator $\mathscr{E}_D$ (see equation \eqref{eq3.26}). 
	Hence, to complete the verification of intertwining relation in equation \eqref{eq3.28}, it remains only to establish that for each index $i=1,\dots,n-1$, the following relations hold
	\begin{equation}\label{eq3.31}
		\mathscr{D}_{D_{S_n^*},S_n^*}S_i^*=M^*_{\Phi_i}\mathscr{D}_{D_{S_n^*},S_n^*},
	\end{equation}
	and
	\begin{equation}\label{eq3.32}
		\mathcal{P}_{S_n^*}S_i^*=D_i^*\mathcal{P}_{S_n^*}.
	\end{equation}
	Since the tuple $(D_1,\dots, D_n)$, acting on $\mathscr{P}_{S_n^*}$, is the canonical $\Gamma_n$-unitary associated with the $\Gamma_n$-contraction $(S_1,\dots,S_n)$, equation \eqref{eq3.32} follows directly from construction. Consequently, it remains to verify equation \eqref{eq3.31}. By simultaneously using \eqref{eq3.24} and Theorem \ref{thm1.3}, we obtain \begin{align*}
		M^*_{\Phi_i}\mathscr{D}_{D_{S_n^*},S_n^*} &=E_i D_{S_n^*}(I-zS_n^*)^{-1}+E_{n-i}^* D_{S_n^*}(I-zS_n^*)^{-1} S_n^*\\ 
		&=(E_iD_{S_n^*}+E_{n-i}^* D_{S_n^*}S_n^*)(I-zS_n^*)^{-1}\\ 
		&=D_{S_n^*}S_i^*(I-zS_n^*)^{-1}\\ 
		&=D_{S_n^*}(I-zS_n^*)^{-1}S_i^*\\ 
		&=\mathscr{D}_{D_{S_n^*},S_n^*} S_i^* ,
	\end{align*}
	for all $i=1,\dots,n-1$. This establishes that the isometric embedding operator $\mathscr{E}_D$ intertwines the operator tuple $(S_1^*,\dots,S_n^*)$ and $(\hat{V_1}^*,\dots,\hat{V_n}^*)$. We now show that the tuple $(\hat{V_1},\dots,\hat{V_n})$ is a $\Gamma_n$-isometry. By assumption, $(\gamma_1M_{\Phi_1},\dots,\gamma_{n-1}M_{\Phi_{n-1}})$ is a $\Gamma_{n-1}$-contraction, and the fundamental operators $E_1,\dots,E_{n-1}$ associated with $(S_1^*,\dots,S_n^*)$ satisfy the commutativity condition \eqref{eq1.1}. Therefore, applying Proposition \ref{prop3.1}, we conclude that the tuple $(M_{\Phi_1},\dots,M_{\Phi_{n-1}},M_z)$ is a $\Gamma_n$-isometry. Since $(D_1,\dots,D_n)$ is the canonical $\Gamma_n$-unitary associated with $(S_1,\dots,S_n)$, Lemma \ref{lem2.2}, implies that the tuple $(\hat{V_1},\dots,\hat{V_n})$ is a $\Gamma_n$ isometry.  Furthermore, if we assume that $\mathcal{H}_D$ is the functional model given by $\mathcal{H}_{D}=\operatorname{Ran}\mathscr{E}_D$, then equation \eqref{eq3.28} implies that $(S_1,\dots,S_n)$ is jointly unitary  equivalent to \begin{equation*}
		P_{\mathcal{H}_D}\left( \begin{bmatrix}
			M_{\Phi_1} & 0\\
			0 & D_1
		\end{bmatrix},\dots,\begin{bmatrix}
			M_{\Phi_{n-1}} & 0\\
			0 & D_{n-1}
		\end{bmatrix},\begin{bmatrix}
			M_z & 0\\
			0 & D_n 
		\end{bmatrix} \right)\Bigg|_{\mathcal{H}_D}.
	\end{equation*}
	
	We now demonstrate that the operator tuple in equation \eqref{eq-modeld} is uniquely determined. For this, assume that \begin{equation*}
		(\tilde{V}_1,\dots,\tilde{V}_{n-1},\hat{V_n}):=\left( \begin{bmatrix}
			\tilde{V}^{(1)}_{11} & \tilde{V}^{(1)}_{12}\\
			\tilde{V}^{(1)}_{21} & \tilde{V}^{(1)}_{22}
		\end{bmatrix},\dots,\begin{bmatrix}
			\tilde{V}^{(n-1)}_{11} & \tilde{V}^{(n-1)}_{12}\\
			\tilde{V}^{(n-1)}_{21} & \tilde{V}^{(n-1)}_{22}
		\end{bmatrix},\begin{bmatrix}
			M_z & 0\\
			0 & D_n 
		\end{bmatrix}\right)
	\end{equation*} 
	acting on $H_{\mathcal{D}_{S_n^*}}^2(\mathbb{D})\oplus
	\mathscr{P}_{S_n^*}$ is the $\Gamma_n$-isometric dilation of the $\Gamma_n$-contraction $(S_1,\dots,S_n)$. Next we proceed as in the Theorem \ref{thm2.3} to obtain the following block diagonal form of the tuple \begin{equation*}
		(\tilde{V}_1,\dots,\tilde{V}_{n-1})=
		\left( \begin{bmatrix}
			M_{\tilde{\Phi}_1} & 0\\
			0 & \tilde{V}^{(1)}_{22}
		\end{bmatrix},\dots,\begin{bmatrix}
			M_{\tilde{\Phi}_{n-1}} & 0\\
			0 & \tilde{V}^{(n-1)}_{22}
		\end{bmatrix}\right)
	\end{equation*} where $(\tilde{V}^{(1)}_{22},\dots,\tilde{V}^{(n-1)}_{22},D_n)$ is a $\Gamma_n$-unitary, $\tilde{\Phi}_i=\tilde{E}_i^*+z\tilde{E}_{n-i} \in H^{\infty}(\mathcal{B}(\mathcal{D}_{S_n^*}) $ for all $i=1,\dots,n-1$, and for some operators $\tilde{E}_1,\dots,\tilde{E}_{n-1}~\in \mathcal{B}(\mathcal D_{S_n^*})$. Furthermore, since the tuple $(\tilde{V}_1,\dots,\tilde{V}_{n-1},\hat{V_n})$ is an isometric dilation of $(S_1,\dots,S_n)$, it satisfies the following intertwining condition 
	$$\begin{bmatrix}
		\mathscr{D}_{D_{S_n^*},S_n^*}\\
		\mathcal{P}_{S_n^*}
	\end{bmatrix}S_i^*=\begin{bmatrix}
		M_{\tilde{\Phi}_i} & 0\\
		0 & {\tilde{V}_{22}}^{(i)}
	\end{bmatrix}^* \begin{bmatrix}
		\mathscr{D}_{D_{S_n^*},S_n^*}\\
		\mathcal{P}_{S_n^*}
	\end{bmatrix},$$
	for all $i=1,\dots,n-1$. Equivalently, we can write
	\begin{equation}\label{eq3.36}
		\mathscr{D}_{D_{S_n^*},S_n^*}(S_1,\dots,S_n)^*=(M_{\tilde{\Phi}_1},\dots,M_{\tilde{\Phi}_{n-1}},M_z)^*\mathscr{D}_{D_{S_n^*},S_n^*}
	\end{equation}
	and
	\begin{equation}\label{eq3.37}
		\mathcal{P}_{S_n^*}(S_1,\dots,S_n)^*=({{\tilde{V}^{(1)}}_{22}},\dots,{{\tilde{V}}^{(n-1)}_{22}},\hat{V_n})^*\mathcal{P}_{S_n^*}.
	\end{equation} 
	First we claim that $({{\tilde{V}^{(1)}}_{22}},\dots,{{\tilde{V}}^{(n-1)}_{22}})=(D_1,\dots,D_{n-1})$. Note that the operator ${\tilde{V}^{(i)}}_{22}$ commutes with $D_n$ for all $i=1,\dots,n-1$. Since $D_n$ is unitary, it follows that ${\tilde{V}^{(i)}}_{22}$ also commutes with $D_n^*$ for all $i=1,\dots,n-1.$ Consequently, ${\tilde{V}^{(i)*}}_{22}$ commutes with $D_n$ for all $i=1,\dots,n-1.$ Utilizing equation \eqref{eq-def}, and \eqref{eq3.37}, we have 
	\begin{align*}
		{{\tilde{V}^{(i)*}}_{22}}(D_n^k\mathcal{P}_{S_n^*}h)&=D_n^k{{\tilde{V}^{(i)*}}_{22}}\mathcal{P}_{S_n^*}h=D_n^k\mathcal{P}_{S_n^*}S_i^*h\\
		&=D_n^kD_i^*\mathcal{P}_{S_n^*}h=D_i^*(D_n^k\mathcal{P}_{S_n^*}h)
	\end{align*}
	for all $i=1,\dots,n-1.$ Since the set $\{D_n^k\mathcal{P}_{S_n^*}h~:~k\geq0~\text{and}~h\in \mathcal{H}\}$ is dense in $\mathscr{P}_{S_n^*}$, it follows that ${\tilde{V}^{(i)}}_{22}=D_i$ for all $i=1,\dots,n-1$. This establishes the desired claim that $({{\tilde{V}^{(1)}}_{22}},\dots,{{\tilde{V}}^{(n-1)}_{22}})=(D_1,\dots,D_{n-1})$. Next we show that $\tilde{E}_1,\dots,\tilde{E}_{n-1}$ are the fundamental operators of $(S_1^*,\dots,S_n^*)$. Invoking the power series expansion of $\mathscr{D}_{D_{S_n^*},S_n^*}$ and applying equation \eqref{eq3.36}, we deduce the identity
	\begin{equation*}
		D_{S_n^*} S_i^* = \tilde{E}_i D_{S_n^*}+ \tilde{E}^*_{n-i} D_{S_n^*} S_n^*
	\end{equation*}
	for all $i=1,\dots,n-1$. This proves that $\tilde{E}_1,\dots,\tilde{E}_{n-1}$ are the fundamental operators for the $\Gamma_n$-contraction $(S_1^*,\dots, S_n^*)$. But the uniqueness of the fundamental operators guarantees that the tuple $(\tilde{E}_1,\dots,\tilde{E}_{n-1})=(E_1,\dots,E_{n-1})$. This completes the proof.
\end{proof}

\section{Factorization Theorem}
In this section, we establish various factorization theorems that facilitate the connection between an arbitrary minimal isometric dilation and an arbitrary isometric dilation of a given contraction.
\begin{thm}\label{thm4.1}
	Let $T$ be a contraction acting on the Hilbert space $\mathcal{H}$. Assume that $V$ acting on the Hilbert space $\mathcal{K}$ is an isometric dilation of $T$ corresponding to an isometry $\mathscr E:\mathcal{H}\rightarrow \mathcal{K}$ and $V_m$ acting on the Hilbert space $\mathcal{K}_m$ be any minimal isometric dilation of $T$ corresponding to an isometry $\mathscr E_m:\mathcal{H}\rightarrow \mathcal{K}_m$. Then there exist a unique isometry $\Xi\in \mathcal{B}(\mathcal{K}_m,\mathcal{K})$ such that 
	\begin{equation*}\label{eq4.38}
		\mathscr E=\Xi\mathscr E_m \quad \text{and} \quad \Xi V_m^*=V^* \Xi.
	\end{equation*} 
	% 	Furthermore, if $\mathcal{K}_1=H^2_{\mathcal{E}}(\mathbb{D})\oplus \mathcal{K}_u$ and $V_1=M_z\oplus U$ be the Wold decomposition of $V_1$ for some unitary $U\in \mathcal{B}(K_u)$. Then 
	% 	\begin{equation}\label{eq4.39}
		% 		\Phi=(I_{H^2(\mathbb{D})}\otimes \tilde{V}_1)\oplus \tilde{V}_2
		% 	\end{equation} 
	% 	for some isometry $\tilde{V}_1\in \mathcal{B}(\mathcal{D}_{\mathcal{P}^*},\mathcal E)$ and isometry $\tilde{V}_2\in \mathcal{B}(\overline{\Delta_PL^2_{D_P}(\mathbb{T})},\mathcal K_u)$.
\end{thm}
\begin{proof}
	Given $V_m$ acting on the Hilbert space $\mathcal{K}_m$ is any minimal isometric dilation of $T$ corresponding to an isometry $\mathscr E_m:\mathcal{H}\rightarrow \mathcal{K}_m$, therefore we have, 
	\begin{equation*}
		\mathcal{K}_m=\bigvee_{k=0}^{\infty}V_m^k(\mathscr E_m\mathcal{H}).
	\end{equation*}
	Also, $V$ acting on the Hilbert space $\mathcal{K}$ is an isometric dilation of $T$ corresponding to an isometry $\mathscr E:\mathcal{H}\rightarrow \mathcal{K}$ then $V$-reducing subspace $$\mathcal{K}_1:=\bigvee_{k=0}^{\infty}V^k(\mathscr E \mathcal{H})\subseteq\mathcal{K}$$ is the minimal isometric dilation space of $T$. Therefore, there exist an isometry $\Xi:\mathcal K_m\rightarrow \mathcal K=\mathcal K_1\oplus \mathcal K_2$ defined by \begin{equation}\label{eq4.39}
		\Xi(V_m^k(\mathscr E_mh))=V^k(\mathscr E h)\quad\text{for all}\quad h\in \mathcal{H}\quad \text{and}\quad k\in \{0\}\cup\mathbb{N},
	\end{equation}
	where $\mathcal{K}_2=\mathcal{K}\ominus\mathcal{K}_1$. This proves the fact that $\mathscr E=\Xi\mathscr E_m$. Next, we prove that $\Xi$ intertwines $V^*$ and $V_m^*$. Observe that, we have the following relation 
	\begin{align*}
		\Xi V_m^*(V_m^k\mathscr E_m)&=\Xi V_m^{k-1}\mathscr E_m\\
		&=V^{k-1}\mathscr E\\
		&=V^*(V^k\mathscr E)=V^*\Xi (V_m^k\mathscr E_m)
	\end{align*}
	for all $k\in \mathbb{N}$ and \begin{equation*}
		\Xi V_m^*\mathscr E_m=\Xi \mathscr E_m T^*=\mathscr E T^*=V^*\mathscr E=V^*\Xi \mathcal E_m.
	\end{equation*}
	This shows that $\Xi V_m^*=V^*\Xi.$ Moreover, using \eqref{eq4.39} one can easily prove that $\Xi$ is unique. 
	%	 We now show that $\Phi$ has the form stated in \eqref{eq4.39}. Assume that $V_1$ and $V_2$ be isometries with Wold decompositions given by $V_1=M^1_z\oplus U_1$ acting on $\mathcal{K}_1=H^2_{E_1}(\mathbb{D})\oplus\mathcal{K}^1_u$ and $V_2=M^2_z\oplus U_2$ acting on $\mathcal{K}_2=H^2_{E_2}(\mathbb{D})\oplus\mathcal{K}^2_u$ where $M^1_z\in \mathcal{B}(H^2_{E_1}(\mathbb{D}))$ and $M^2_z\in \mathcal{B}(H^2_{E_2}(\mathbb{D}))$ are pure shifts, and $U_1\in \mathcal{B}(\mathcal{K}^1_u)$ and $U_2\in \mathcal{B}(\mathcal{K}^2_u)$ are unitaries. Let 
	% 	\begin{equation}
		% 		\Phi=\begin{bmatrix}
			% 			X_1 & X_2\\
			% 			X_3 & X_4
			% 		\end{bmatrix}: \mathcal{K}_2=H^2_{E_2}(\mathbb{D})\oplus\mathcal{K}^2_u\rightarrow \mathcal{K}_1=H^2_{E_1}(\mathbb{D})\oplus\mathcal{K}^1_u
		% 	\end{equation}
	% 	Then, using the intertwining property of $\Phi$ as mentioned in \eqref{eq4.38}, we have $$X_1M^{2*}_z=M^{1*}_zX_1,~ X_2U_2^*=M^{1*}_zX_2,~ X_3M^{2*}_z=U_1^*X_3~\text{and}~X_4U_2^*=U_1^*X_4.$$ Note that $X_2^*$ and $X_3^*$ intertwines a unitary and a pure isometry, therefore, it follows (see Lemma 2.11) that $X_2=0$ and $X_3=0$. Hence, $$\Phi=\begin{bmatrix}
		% 	X_1 & 0\\
		% 	0 & V_2
		% 	\end{bmatrix},$$ 
	% 	where $X_4=V_2$, and$$\ran\Phi=\ran X_1\oplus \ran V_2\subseteq \mathcal{K}_1=H^2_{E_1}(\mathbb{D})\oplus\mathcal{K}^1_u$$   
	This completes the proof.
\end{proof}
\begin{cor}\label{cor4.2}
	Let $T$ be a contraction acting on the Hilbert space $\mathcal{H}$. Assume that $V_{m_1}$ acting on the Hilbert space $\mathcal{K}_{m_1}$ be a minimal isometric dilation of $T$ corresponding to an isometry $\mathscr E_{m_1}:\mathcal{H}\rightarrow \mathcal{K}_{m_1}$ and $V_{m_2}$ acting on the Hilbert space $\mathcal{K}_{m_2}$ be any other minimal isometric dilation of $T$ corresponding to an isometry $\mathscr E_{m_2}:\mathcal{H}\rightarrow \mathcal{K}_{m_2}$. Then there exist a unique unitary $\Xi_u\in \mathcal{B}(\mathcal{K}_{m_2},\mathcal{K}_{m_1})$ such that 
	\begin{equation*}\label{eq4.40}
		\mathscr E_{m_1}=\Xi_u\mathscr E_{m_2} \quad \text{and} \quad \Xi_u V_{m_2}^*=V_{m_1}^*\Xi_u.
	\end{equation*} 
\end{cor}
Note that the above factorization results can be visualized through the following commutative diagrams: 
\begin{figure}[h]
	\begin{subfigure}{0.45\textwidth}
		\centering 		
		\begin{tikzcd}\label{fig1}
			\mathcal{H} \arrow{r}{\mathscr E} \arrow[swap]{dr}{\mathscr E_{m}} & \mathcal{K} \\
			& \arrow[swap]{u}{\Xi} \mathcal{K}_{m}
		\end{tikzcd}
		\caption{Figure 1.1}
	\end{subfigure}
	\hfill
	\begin{subfigure}{0.45\textwidth}
		\centering
		\begin{tikzcd}\label{fig2}
			\mathcal{H} \arrow{r}{\mathscr E_{m_1}} \arrow[swap]{dr}{\mathscr E_{m_2}} & \mathcal{K}_{m_1} \\
			& \arrow[swap]{u}{\Xi_u}\mathcal{K}_{m_2}
		\end{tikzcd}.
		\caption{Figure 1.2}
	\end{subfigure}
\end{figure}

\begin{rem}\label{rem4.3}
	If $T$ is a (c.n.u) contraction on a Hilbert space $\mathcal{H}$. It is well known (Douglas \cite{10}; Sz.-Nagy and Foias \cite{Nagy}) that $T$ admits a minimal isometric dilation. Let $V_D$ and $V_{NF}$ be the Douglas minimal isometric dilation and Sz.-Nagy and Foias minimal isometric dilation, respectively. Note that each of these dilations admits a well-defined explicit form. Denote by $\mathscr E_{D}:\mathcal{H}\to\mathcal{K}_D$ and $\mathscr E_{NF}:\mathcal{H}\to\mathcal{K}_{NF}$ the embeddings associated with the minimal dilations, respectively. Suppose that $V$ is any isometric dilation of $T$ on a Hilbert space $\mathcal{K}$ with isometric embedding $\mathscr E:\mathcal{H}\to\mathcal{K}$. By virtue of Theorem~\ref{thm4.1}, there exist isometries $\Xi_{1}\in\mathcal{B}(\mathcal{K}_{D},\mathcal{K})$ and $\Xi_{2}\in\mathcal{B}(\mathcal{K}_{NF},\mathcal{K})$ such that \begin{equation}\label{eq4.41} 		
		\mathscr E=\Xi_1\mathscr E_{D} \quad \text{and} \quad \Xi_1 V_D^*=V^* \Xi_1, 
	\end{equation} 
	and  
	\begin{equation}\label{eq4.42} 		
		\mathscr E=\Xi_2\mathscr E_{NF} \quad \text{and} \quad \Xi_2 V_{NF}^*=V^* \Xi_2. 
	\end{equation} 
	Moreover, using Corollary \ref{cor4.2} it follows that there exist a unique unitary $\Xi_u\in \mathcal{B}(\mathcal{K}_{NF},\mathcal{K}_D)$ such that 
	\begin{equation}\label{eq4.43}
		\mathscr E_{D}=\Xi_u\mathscr E_{NF}\quad \text{and} \quad \Xi_u V_{NF}^*=V_{D}^* \Xi_u.
	\end{equation}
\end{rem}

In 2015, Sarkar \cite{17} obtained a factorization theorem that elucidates the relationship between Sz.-Nagy-Foias minimal isometric dilation and an arbitrary isometric dilation of a given contraction. The above equations in \eqref{eq4.41} and \eqref{eq4.42} describes how each minimal isometric dilation (Douglas; Sz.-Nagy-Foias) is related to an arbitrary isometric dilation of the contraction $T$. Also, equation \eqref{eq4.43} highlights the interplay between Douglas minimal isometric dilation and Sz.-Nagy-Foias minimal isometric dilation. In fact, the first part of equations in \eqref{eq4.41}, \eqref{eq4.42}, and \eqref{eq4.43} admit the following commutative diagram representations.
\begin{figure}[h]
	\begin{subfigure}{0.45\textwidth}
		\centering 
		\begin{tikzcd}\label{fig3}
			\mathcal{H} \arrow{r}{\mathscr E} \arrow[swap]{dr}{\mathscr E_{D}} & \mathcal{K} \\
			& \arrow[swap]{u}{\Xi_1} \mathcal{K}_{D}
		\end{tikzcd}
		\caption{Figure 1.3}
		\label{fig3}
	\end{subfigure}
	\begin{subfigure}{0.45\textwidth}
		\centering 
		\begin{tikzcd}\label{fig4}
			\mathcal{H} \arrow{r}{\mathscr E} \arrow[swap]{dr}{\mathscr E_{NF}} & \mathcal{K} \\
			& \arrow[swap]{u}{\Xi_2} \mathcal{K}_{NF}
		\end{tikzcd}
		\caption{Figure 1.4}
		\label{fig4}
	\end{subfigure}
	\begin{subfigure}{0.45\textwidth}
		\centering 
		\begin{tikzcd}\label{fig5}
			\mathcal{H} \arrow{r}{\mathscr E_{D}} \arrow[swap]{dr}{\mathscr E_{NF}} & \mathcal{K}_D \\
			& \arrow[swap]{u}{\Xi_u} \mathcal{K}_{NF}
		\end{tikzcd}
		\caption{Figure 1.5}
	\end{subfigure}
\end{figure}

\section{A Sz.-Nagy-Foias and Sch\"affer Type Functional Model}
In this section, we develop two functional models for $\Gamma_n$-contractions: the first is derived from the Sz.-Nagy-Foias model theory, and the second is obtained from the Sch\"affer model for a contraction. 
\subsection{Sz.-Nagy-Foias Type Functional Model}
Sz.-Nagy and Foias developed an explicit functional model for the minimal isometric dilation of a completely non-unitary (c.n.u.) contraction (see~\cite{Nagy}). They introduced the \textit{characteristic function} of a contraction $T$ on a Hilbert space $\mathcal{H}$, which is a contractive analytic function on the unit disk $\mathbb{D}$ and is given explicitly by  
\begin{equation*}
	\Theta_T(z):=[-T+z\mathscr{P}_{D_{T^*},T^*}T] |_{\mathcal {D}_T}:\mathcal {D}_T\rightarrow \mathcal {D}_{T^*}.
\end{equation*}
Let $\Theta_T(\eta)$ denote the radial limit of the characteristic function, defined almost everywhere on $\mathbb{T}$, and let $\Delta_T(\eta)$ denote its defect operator, given by 
\begin{equation*}
	\Delta_T(\eta)=(I-\Theta_T(\eta)^*\Theta_T(\eta))^\frac{1}{2}.
\end{equation*}
In this setting, the operator 
$$\begin{bmatrix}
	M_z & 0\\
	0 & M_{\eta}|_{\overline{\Delta_{T}L^2_{\mathcal{D}_{T}}(\mathbb{T})}}
\end{bmatrix},$$
acting on the dilation space $H^2_{\mathcal{D}_{T^*}}(\mathbb{D}) \oplus
\overline{\Delta_{T}L^2_{\mathcal{D}_T}(\mathbb{T})}$, is an isometry. Moreover, by \cite[Theorem 2.5]{17} there exists an isometry $\mathscr{E}_{NF}:\mathcal{H}\rightarrow H^2_{\mathcal{D}_{T^*}}(\mathbb{D}) \oplus
\overline{\Delta_{T}L^2_{\mathcal{D}_T}(\mathbb{T})}$ that satisfies the intertwining relation $$\mathscr{E}_{NF}T^*=\begin{bmatrix}
	M_z & 0\\
	0 & M_{\eta}|_{\overline{\Delta_{T}L^2_{\mathcal{D}_{T}}(\mathbb{T})}} 
\end{bmatrix}^*\mathscr{E}_{NF}.$$
The next theorem develops a Sz.-Nagy-Foias type functional model for the 
$\Gamma_n$-contraction $(S_1,\dots,S_n)$, thus generalizing the classical dilation framework to several variables.

\begin{thm}
	\textit{(Sz.-Nagy-Foias Type Functional Model.)} Let $(S_1,\dots,S_n)$ be a $\Gamma_n$-contraction acting on a Hilbert space $\mathcal{H}$ such that $S_n$ is c.n.u, and $E_1,\dots,E_{n-1}$ be the fundamental operators of $(S_1^*,\dots,S_n^*)$ on $\mathcal{H}$ that satisfies the commutativity condition \eqref{eq1.1}. Let $\Phi_1,\dots,\Phi_{n-1}\in H^{\infty}(\mathcal{B}(\mathcal{D}_{S_n^*}))$ be obtained as in Theorem \ref{thm2.3} such that $(\gamma_1 M_{\Phi_1}, \dots, \gamma_{n-1} M_{\Phi_{n-1}})$ is a $\Gamma_{n-1}$-contraction. If $(D_1,\dots,D_n)$ be the canonical $\Gamma_n$-unitary constructed from $(S_1,\dots,S_n)$. Then $(S_1,\dots,S_n)$ is jointly unitarily equivalent to \begin{equation*}
		P_{\mathcal{H}_{NF}}\left( \begin{bmatrix}
			M_{\Phi_1} & 0\\
			0 & N_1
		\end{bmatrix},\dots,\begin{bmatrix}
			M_{\Phi_{n-1}} & 0\\
			0 & N_{n-1}
		\end{bmatrix},\begin{bmatrix}
			M_z & 0\\
			0 & M_{\eta}|_{\overline{\Delta_{S_n}L^2_{\mathcal{D}_{S_n}}(\mathbb{T})}}
		\end{bmatrix}\right)\Bigg|_{\mathcal{H}_{NF}},
	\end{equation*}
	where $(N_1,\dots,N_{n-1})=U^*(D_1,\dots,D_{n-1})U$, from some unitary $U:\overline{\Delta_{S_n}L^2_{\mathcal{D}_{S_n}}(\mathbb{T})}\rightarrow \mathscr{P}_{S_n^*}$, and $\mathcal{H}_{NF}$ is the functional model space given by 
	\begin{equation*}
		\mathcal{H}_{NF}=\operatorname{Ran} \mathcal{E}_{NF}=\begin{bmatrix}
			H^2_{\mathcal{D}_{S_n^*}}(\mathbb{D})\\
			\overline{\Delta_{S_n}L^2_{\mathcal{D}_{S_n}}(\mathbb{T})}
		\end{bmatrix}\ominus
		\begin{bmatrix}
			\Theta_{S_n}\\
			\Delta_{S_n}
		\end{bmatrix}.H^2_{\mathcal{D}_{S_n}}(\mathbb{D}).
	\end{equation*}
\end{thm}
\begin{proof}
	We use Douglas' minimal isometric dilation to obtain the Sz.-Nagy and Foias type functional model. Given that $(S_1,\dots,S_{n-1},S_n)$ is a $\Gamma_n$-contraction and $S_n$ is c.n.u. Hence using equation \eqref{eq4.43} of Remark \ref{rem4.3}, there exist a unique unitary $\Xi_u\in \mathcal{B}\left(\mathcal{K}_{NF},\mathcal{K}_{D}\right)$ such that $\mathscr{E}_D=\Xi_{u}\mathscr{E}_{NF}$ and
	\begin{equation}\label{eqrel-nd}
		\Xi_{u} \begin{bmatrix}
			M_z & 0\\
			0 & M_{\eta}|_{\overline{\Delta_{S_n}L^2_{\mathcal{D}_{S_n}}(\mathbb{T})}}
		\end{bmatrix}^*=\begin{bmatrix}
			M_z & 0\\
			0 & D_n
		\end{bmatrix}^*\Xi_{u},
	\end{equation}
	$$\text{where} \quad \mathscr{E}_D:\mathcal{H}:\rightarrow	\begin{bmatrix}
		H^2_{\mathcal{D}_{S_n^*}}(\mathbb{D}) \\
		\mathscr{P}_{S_n^*}
	\end{bmatrix}:=\mathcal{K}_D~ \text{and}~ \mathscr{E}_{NF}:\mathcal{H}:\rightarrow\begin{bmatrix}
		H^2_{\mathcal{D}_{S_n^*}}(\mathbb{D})\\
		\overline{\Delta_{S_n}L^2_{\mathcal{D}_{S_n}}(\mathbb{T})}
	\end{bmatrix}:=\mathcal{K}_{NF}$$
	are the isometries associated with the Douglas minimal isometric dilation and Sz.-Nagy and Foias minimal isometric dilation, respectively. Now, we obtain the exact form of the operator $\Xi_u$. For this, assume that $\Xi_u=\begin{bmatrix}
		U_{11} & U_{12}\\
		U_{21} & U_{22}
	\end{bmatrix}$ and substituting it in equation \eqref{eqrel-nd}, we obtain that \begin{equation*}
		\begin{bmatrix}
			U_{11}M_z^* & U_{12}{M_{\eta}|_{\overline{\Delta_{S_n}L^2_{\mathcal{D}_{S_n}}(\mathbb{T})}}}^*\\
			U_{21}M_z^* & U_{22}{M_{\eta}|_{\overline{\Delta_{S_n}L^2_{\mathcal{D}_{S_n}}(\mathbb{T})}}}^*
		\end{bmatrix}=\begin{bmatrix}
			M_z^*U_{11} & M_z^*U_{12}\\
			D_n^*U_{21} & D_n^*U_{22}
		\end{bmatrix}. 
	\end{equation*} Equating the $(1,2)$-th entry and $(2,1)$-th entry, and later applying \cite[Lemma 2.5]{ay1} implies that $U_{12}=U_{21}=0$. Finally, by equating the $(1,1)$-th entry we have $U_{11}=M_{\Psi}$ for some $\Psi\in H^{\infty}(\mathcal{B}(\mathcal{D}_{S_n^*}))$. Moreover, the fact that $\Xi_u$ is unitary, forces $M_{\Psi}$ to be equal to the identity operator in $H^2_{\mathcal{D}_{S_n^*}}(\mathbb{D})$. Hence, the unitary operator $\Xi_u$ has a specific form given by 
	\begin{equation*}
		\Xi_u=\begin{bmatrix}
			I_{H^2_{\mathcal{D}_{S_n^*}}(\mathbb{D})} & 0\\
			0 & U
		\end{bmatrix},
	\end{equation*}
	for some unitary $U:\overline{\Delta_{S_n}L^2_{\mathcal{D}_{S_n}}(\mathbb{T})}\rightarrow \mathscr{P}_{S_n^*}$.
	Now, as a consequence of the assumption and Theorem \ref{thm3.2}, we get the following relation 
	\begin{equation}\label{eq5.50}
		\mathscr{E}_D(S_1,\dots,S_{n-1},S_n)^*=(\hat{V}_1,\dots,\hat{V}_{n-1},\hat{V}_n)^*\mathscr{E}_D,
	\end{equation}
	where
	$$(\hat{V}_1,\dots,\hat{V}_{n-1},\hat{V}_n)=\left( \begin{bmatrix}
		M_{\Phi_1} & 0\\
		0 & D_1
	\end{bmatrix},\dots,\begin{bmatrix}
		M_{\Phi_{n-1}} & 0\\
		0 & D_{n-1}
	\end{bmatrix},\begin{bmatrix}
		M_z & 0\\
		0 & D_n 
	\end{bmatrix} \right)$$
	and $(D_1,\dots,D_n)$ be the $\Gamma_n$-unitary canonically constructed from the $\Gamma_n$-contraction $(S_1,\dots,S_n)$. Substituting the expression of $\mathscr E_D$ by $\Xi_u\mathscr{E}_{NF}$ in equation \eqref{eq5.50} we get
	\begin{equation*}
		\Xi_u\mathscr{E}_{NF}(S_1,\dots,S_{n-1},S_n)^*=(\hat{V}_1,\dots,\hat{V}_{n-1},\hat{V}_n)^*\Xi_u\mathscr{E}_{NF}.
	\end{equation*}
	It is equivalent to 
	\begin{equation}\label{eq5.51}
		\mathscr{E}_{NF}(S_1,\dots,S_{n-1},S_n)^*=(W_1,\dots,W_{n-1},W_n)^*\mathscr{E}_{NF}.
	\end{equation}  
	Clearly $(W_1,\dots,W_{n-1},W_n):=\Xi_u^*(\hat{V_1},\dots,\hat{V_n})\Xi_u$ is $\Gamma_n$-isometry (see Lemma \ref{lem2.1}) and is equal to 
	\begin{equation*}
		\left( \begin{bmatrix}
			M_{\Phi_1} & 0\\
			0 & U^*D_1U
		\end{bmatrix},\dots,\begin{bmatrix}
			M_{\Phi_{n-1}} & 0\\
			0 & U^*D_{n-1}U
		\end{bmatrix},\Xi_u^*\begin{bmatrix}
			M_z & 0\\
			0 & D_n 
		\end{bmatrix}\Xi_u \right). 
	\end{equation*} At this point, we use the uniqueness of the Sz.-Nagy-Foias minimal isometric dilation for the c.n.u $S_n$, which guarantees that $(W_1,\dots,W_{n-1},W_n)$ is equal to $$\left( \begin{bmatrix}
		M_{\Phi_1} & 0\\
		0 & N_1
	\end{bmatrix},\dots,\begin{bmatrix}
		M_{\Phi_{n-1}} & 0\\
		0 & N_{n-1}
	\end{bmatrix},\begin{bmatrix}
		M_z & 0\\
		0 & M_{\eta}|_{\overline{\Delta_{S_n}L^2_{\mathcal{D}_{S_n}}(\mathbb{T})}}
	\end{bmatrix} \right),$$
	where $(N_1,\dots,N_{n-1})=U^*(D_1,\dots,D_{n-1})U$. Furthermore, if we assume that $\mathcal{H}_{NF}$ is the functional model space given by 
	\begin{equation*}
		\mathcal{H}_{NF}=\operatorname{Ran} \mathscr{E}_{NF}=\begin{bmatrix}
			H^2_{\mathcal{D}_{S_n^*}}(\mathbb{D})\\
			\overline{\Delta_{S_n}L^2_{\mathcal{D}_{S_n}}(\mathbb{T})}
		\end{bmatrix}\ominus
		\begin{bmatrix}
			\Theta_{S_n}\\
			\Delta_{S_n}
		\end{bmatrix}.H^2_{\mathcal{D}_{S_n}}(\mathbb{D}),
	\end{equation*}then equation \ref{eq5.51} yields that $(S_1,\dots,S_n)$ is jointly unitary equivalent to
	\begin{equation*}
		P_{\mathcal{H}_{NF}}\left( \begin{bmatrix}
			M_{\Phi_1} & 0\\
			0 & N_1
		\end{bmatrix},\dots,\begin{bmatrix}
			M_{\Phi_{n-1}} & 0\\
			0 & N_{n-1}
		\end{bmatrix},\begin{bmatrix}
			M_z & 0\\
			0 & M_{\eta}|_{\overline{\Delta_{S_n}L^2_{\mathcal{D}_{S_n}}(\mathbb{T})}}
		\end{bmatrix}\right)\Bigg|_{\mathcal{H}_{NF}}.
	\end{equation*}
	This completes the proof.
\end{proof}

\subsection{Sch\"affer Type Functional Model}
A few years after Sz.-Nagy established his dilation theorem, Sch\"affer introduced in \cite{16} the first explicit construction of a minimal isometric dilation for a contraction. We begin by revisiting the Sch\"affer isometric dilation corresponding to a contraction $T$ acting on a Hilbert space $\mathcal{H}$. In this framework, the operator $$\begin{bmatrix}
	T & 0\\
	\textbf{D}_T & M_z 
\end{bmatrix}$$ defined on the dilation space $\mathcal{H}\oplus H^2_{\mathcal{D}_T}(\mathbb{D})$, serves as an isometry, where the map  $\textbf{D}_T:\mathcal{H}\rightarrow H^2_{\mathcal{D}_T}(\mathbb{D})$ denotes the constant function given by $$(\textbf{D}_Th)(z)=D_Th.$$ Furthermore, the map $\mathscr{E}_{Sc}:\mathcal{H}\rightarrow \mathcal{H} \oplus H^2_{\mathcal{D}_T}(\mathbb{D})$, defined by $\mathscr{E}_{Sc}h=h\oplus 0$ for all $h \in \mathcal{H}$ (the reader is referred to \cite[discussion in Page 864]{17}), is an isometry. It satisfies the following intertwining relation $$\mathscr{E}_{Sc}T^*=\begin{bmatrix}
	T & 0\\
	\textbf{D}_T & M_z 
\end{bmatrix}^*\mathscr{E}_{Sc}.$$
The theorem that follows presents a Sch\"affer-type functional model for the 
$\Gamma_n$-contraction $(S_1,\dots,S_n)$, extending this classical dilation framework to the multivariable setting.

\begin{thm}
	\textit{(Sch\"affer Type Functional model.)} 
	Assume that $(S_1,\dots,S_n)$ is a $\Gamma_n$-contraction acting on a Hilbert space $\mathcal{H}$, and let $(F_1,\dots,F_{n-1})$, $(E_1,\dots,E_{n-1})$ are the fundamental operator tuples of $(S_1,\dots,S_n)$, $(S_1^*,\dots,S_n^*)$  respectively.  Let $\Phi_1,\dots,\Phi_{n-1}\in H^{\infty}(\mathcal{B}(\mathcal{D}_{S_n^*}))$ be obtained as in Theorem \ref{thm2.3} such that the tuple $(\gamma_1 M_{\Phi_1}, \dots, \gamma_{n-1} M_{\Phi_{n-1}})$ is a $\Gamma_{n-1}$-contraction. If $(D_1,\dots,D_n)$ be the $\Gamma_n$-unitary canonically constructed from $(S_1,\dots,S_n)$. Then $(S_1,\dots,S_n)$ is jointly unitarily equivalent to
	\begin{equation*}
		P_{\mathcal{H}_{Sc}}\left(\begin{bmatrix}
			S_1 & 0\\
			\textbf{F}_{n-1}^*\textbf{D}_{S_n} & M_{\Psi_1}
		\end{bmatrix},\dots,\begin{bmatrix}
			S_{n-1} & 0\\
			\textbf{F}_1^*\textbf{D}_{S_n} & M_{\Psi_{n-1}}
		\end{bmatrix},\begin{bmatrix}
			S_n & 0\\
			\textbf{D}_{S_n} & M_z 
		\end{bmatrix}\right)\Bigg|_{\mathcal{H}_{Sc}}
	\end{equation*}
	where $\mathcal{H}_{Sc}$ is the functional model space given by $\mathcal{H}_{Sc}=\operatorname{Ran}\mathscr{E}_{Sc}$, $M_{\Psi_i}$ is the multiplication operator on $H^2_{\mathcal{D}_{S_n}}(\mathbb{D})$ corresponding to  $\Psi_i(z)=F_i+zF_{n-i}^*,~F_i\in \mathcal{B}(\mathcal{D}_{S_n})$, and each operator $\textbf{F}_i^*\textbf{D}_{S_n}:\mathcal{H}\rightarrow H^2_{\mathcal{D}_{S_n}}(\mathbb{D})$ is defined by $$(\textbf{F}_i^*\textbf{D}_{S_n}h)(z)=F_i^*D_{S_n}h, \quad z\in \mathbb{D}.$$
\end{thm}

\begin{proof}
	We begin by highlighting the relationship between Douglas minimal isometric dilation and Sch\"affer minimal isometric dilation. By Corollary \ref{cor4.2}, we have $\mathscr{E}_D=\Xi_u^{\prime}\mathscr{E}_{Sc}$, and $$\Xi_u^{\prime}\begin{bmatrix}
		S_n & 0\\
		\textbf{D}_{S_n} & M_z 
	\end{bmatrix}^*=\begin{bmatrix}
		M_z & 0\\
		0 & D_n
	\end{bmatrix}^*\Xi_u^{\prime},$$ 
	where $\Xi_u^{\prime} \in \mathcal{B}(\mathcal{K}_{Sc},\mathcal{K}_D)$ is an unitary operator, and $$\mathcal{K}_{Sc}=
		\mathcal{H} \oplus
		H^2_{\mathcal{D}_{S_n}}(\mathbb{D})
	,\quad \text{and}\quad\mathcal{K}_D=
		H^2_{\mathcal{D}_{S_n^*}}(\mathbb{D}) \oplus
		\mathscr{P}_{S_n^*}.$$ By assumption, $\Phi_1,\dots,\Phi_{n-1}\in H^{\infty}(\mathcal{B}(\mathcal{D}_{S_n^*}))$ are the bounded analytic functions obtained in Theorem \ref{thm2.3} such that the tuple $(\gamma_1 M_{\Phi_1}, \dots, \gamma_{n-1} M_{\Phi_{n-1}})$ is $\Gamma_{n-1}$-contraction. If $(D_1,\dots,D_n)$ denotes the $\Gamma_n$-unitary canonically constructed from $(S_1,\dots,S_n)$, then, by invoking Theorem \ref{thm3.2}, we have the following relation 
	\begin{equation}\label{eq5.54}
		\mathscr{E}_D(S_1,\dots,S_{n-1},S_n)^*
		=(\hat{V}_1,\dots,\hat{V}_{n-1},\hat{V}_n)^*\mathscr{E}_D,
	\end{equation}
	where
	$$(\hat{V}_1,\dots,\hat{V}_{n-1},\hat{V}_n)=\left( \begin{bmatrix}
		M_{\Phi_1} & 0\\
		0 & D_1
	\end{bmatrix},\dots,\begin{bmatrix}
		M_{\Phi_{n-1}} & 0\\
		0 & D_{n-1}
	\end{bmatrix},\begin{bmatrix}
		M_z & 0\\
		0 & D_n 
	\end{bmatrix} \right).$$
	Substituting the expression for $\mathscr E_D$ from above into equation \eqref{eq5.54}, we get
	\begin{equation}\label{eq5.55}
		\mathscr{E}_{Sc}(S_1,\dots,S_{n-1},S_n)^*=(\hat{W}_1,\dots,\hat{W}_{n-1},\hat{W}_n)^*\mathscr{E}_{Sc},
	\end{equation}
	where
	\begin{equation*}\label{eq5.56}
		(\hat{W}_1,\dots,\hat{W}_{n-1},\hat{W}_n)=\Xi_u^{\prime*}(\hat{V}_1,\dots,\hat{V}_{n-1},\hat{V}_n)\Xi_u^{\prime}.
	\end{equation*}
	Now, Lemma \ref{lem2.1} immediately implies that the tuple $(\hat{W}_1,\dots,\hat{W}_{n-1},\hat{W}_n)$ is a $\Gamma_n$-isometry. Furthermore, 
	%Equation \eqref{eq5.55} tells us that $(V_1,\dots,V_{n-1},V_n)$ is Schaffer minimal isometric dilation. B
	by the uniqueness of the Sch\"affer minimal isometric dilation for the contraction $S_n$, it follows that $$\hat{W}_n=\begin{bmatrix}
		S_n & 0\\
		\textbf{D}_{S_n} & M_z 
	\end{bmatrix}.$$
	From equation \eqref{eq5.55}, we obtain $\mathscr{E}_{Sc}S_i^*=\hat{W}_i^*\mathscr{E}_{Sc}$ for all $i=1,\dots,n-1$. This intertwining relation implies that $\hat{W}_i$ is of the form $\begin{bmatrix}
		S_i & 0\\
		\hat{W}^{(i)}_{21} & \hat{W}^{(i)}_{22} 
	\end{bmatrix}$ for all $i=1,\dots,n-1$. Consequently, $$\left( \begin{bmatrix}
		S_1 & 0\\
		\hat{W}^{(1)}_{21} & \hat{W}^{(1)}_{22}
	\end{bmatrix},\dots,\begin{bmatrix}
		S_{n-1} & 0\\
		\hat{W}^{(n-1)}_{21} & \hat{W}^{(n-1)}_{22}
	\end{bmatrix},\begin{bmatrix}
		S_n & 0\\
		\textbf{D}_{S_n} & M_z\end{bmatrix} \right)$$ is a $\Gamma_n$-isometry. At this juncture, using the property of $\Gamma_n$-isometry we know that $\hat{W}_i=\hat{W}_{n-i}^*\hat{W}_n$ for all $i=1,\dots,n-1$. Therefore, 
	\begin{align}\label{eq5.47}
		\begin{bmatrix}
			S_i & 0\\
			\hat{W}^{(i)}_{21} & \hat{W}^{(i)}_{22} 
		\end{bmatrix}\nonumber &=\begin{bmatrix}
			S_{n-i} & 0\\
			\hat{W}^{(n-i)}_{21} & \hat{W}^{(n-i)}_{22}
		\end{bmatrix}^*\begin{bmatrix}
			S_n & 0\\
			\textbf{D}_{S_n} & M_z 
		\end{bmatrix}\\ \nonumber
		&=\begin{bmatrix}
			S_{n-i}^* & \hat{W}^{(n-i)*}_{21}\\
			0 & \hat{W}^{(n-i)*}_{22} 
		\end{bmatrix}\begin{bmatrix}
			S_n & 0\\
			\textbf{D}_{S_n} & M_z 
		\end{bmatrix}\\
		&=\begin{bmatrix}
			S_{n-i}^*S_n+\hat{W}^{(n-i)*}_{21}\textbf{D}_{S_n} & \hat{W}^{(n-i)*}_{21}M_z\\
			\hat{W}^{(n-i)*}_{22}\textbf{D}_{S_n} & \hat{W}^{(n-i)*}_{22}M_z
		\end{bmatrix},
	\end{align}
	for all $i=1,\dots,n-1$. Moreover, 
	\begin{align*}
		\begin{bmatrix}
			S_i & 0\\
			\hat{W}^{(i)}_{21} & \hat{W}^{(i)}_{22} 
		\end{bmatrix}\begin{bmatrix}
			S_n & 0\\
			\textbf{D}_{S_n} & M_z 
		\end{bmatrix}&=\begin{bmatrix}
			S_n & 0\\
			\textbf{D}_{S_n} & M_z 
		\end{bmatrix}\begin{bmatrix}
			S_i & 0\\
			\hat{W}^{(i)}_{21} & \hat{W}^{(i)}_{22} 
		\end{bmatrix},
	\end{align*}
	for all $i=1,\dots,n-1$. It is equivalent to
	\begin{equation}\label{eq5.48}
		\begin{bmatrix}
			S_iS_n & 0\\
			\hat{W}^{(i)}_{21}S_n+\hat{W}^{(i)}_{22}\textbf{D}_{S_n} & \hat{W}^{(i)}_{22}M_z
		\end{bmatrix}=\begin{bmatrix}
			S_nS_i & 0\\
			\textbf{D}_{S_n}S_i+M_z\hat{W}^{(i)}_{21} & M_z \hat{W}^{(i)}_{22}
		\end{bmatrix},
	\end{equation}
	for all $i=1,\dots,n-1$. By comparing the $(2,2)$-th entry in equation \eqref{eq5.48} and have $\hat{W}^{(i)}_{22}M_z=M_z \hat{W}^{(i)}_{22}$ for all $i=1,\dots,n-1$. Hence, each operator $\hat{W}^{(i)}_{22} \in \mathcal{B}(H^2_{\mathcal{D}_{S_n}}(\mathbb{D}))$ commutes with $M_z$. As a result, there exists a bounded analytic function $\Psi_i:\mathbb{D}\rightarrow \mathcal{B}(\mathcal{D}_{S_n})$ such that $\hat{W}^{(i)}_{22}=M_{\Psi_i},$ for all $i=1,\dots,n-1$. Again, equating $(2,2)$-th entry of equation \eqref{eq5.47} and then using $\hat{W}^{(i)}_{22}=M_{\Psi_i}$, we get  $$M_{\Psi_i}=M_{\Psi_{n-i}}^*M_z$$ for all $i=1,\dots,n-1$. This gives that each $\Psi_i$ is of the form $G_i+zG_{n-i}^*$ for some $G_i,G_{n-i}\in \mathcal{B}(\mathcal{D}_{S_n})$. Next, equating $(2,1)$-th entry of equation \eqref{eq5.47}, we obtain 
	\begin{equation*}
		\hat{W}^{(i)}_{21}=M_{\Psi_{n-i}}^*\textbf{D}_{S_n}=(G_{n-i}+M_zG_i^*)^*\textbf{D}_{S_n}=\textbf{G}_{n-i}^*\textbf{D}_{S_n}. 
	\end{equation*}
	Thus $W_i$ admits the representation 
	$$W_i=\begin{bmatrix}
		S_i & 0\\
		\textbf{G}_{n-i}^*\textbf{D}_{S_n} & M_{\Psi_i} 
	\end{bmatrix},$$
	for all $i=1,\dots n-1$. At this stage, we establish that the tuple $(G_1,\dots,G_{n-1})$ coincides with the fundamental operator tuple $(F_1,\dots,F_{n-1})$. To show this we equate the $(1,1)$-entry of equation \eqref{eq5.47}, which yields $S_i=S_{n-i}^*S_n+\hat{W}^{(n-i)*}_{21}\textbf{D}_{S_n}$ for all $i=1,\dots,n-1$. Applying Theorem \ref{thm1.2}, we then have 
	\begin{equation*}
		D_{S_n}F_iD_{S_n}=S_i-S_{n-i}^*S_n=\hat{W}^{(n-i)*}_{21}\textbf{D}_{S_n}=D_{S_n}G_iD_{S_n}.
	\end{equation*}
	Thus, $D_{S_n}\tilde{F}_iD_{S_n}=0$, where $\tilde{F}_i=F_i-G_i$. Indeed, 
	$$\langle \tilde{F}_iD_{S_n}h,D_{S_n}h^{\prime}\rangle=\langle D_{S_n}\tilde{F}_iD_{S_n}h,h^{\prime}\rangle=0,$$ for all $h,h^{\prime}\in \mathcal{H}$, which forces $\tilde{F}_i=0$. Hence $F_i=G_i$ for every $i=1,\dots,n-1$. This proves the fact that 
	\begin{equation}\label{eq5.61}
		\mathscr{E}_{Sc}(S_1,\dots,S_{n-1},S_n)^*=(\hat{W}_1,\dots,\hat{W}_{n-1},\hat{W}_n)^*\mathscr{E}_{Sc},
	\end{equation}
	where the tuple $(\hat{W}_1,\dots,\hat{W}_{n-1},\hat{W}_n)$ is given by  
	\begin{equation*}
		\left(\begin{bmatrix}
			S_1 & 0\\
			\textbf{F}_{n-1}^*\textbf{D}_{S_n} & M_{\Psi_1}
		\end{bmatrix},\dots,\begin{bmatrix}
			S_{n-1} & 0\\
			\textbf{F}_1^*\textbf{D}_{S_n} & M_{\Psi_{n-1}}
		\end{bmatrix},\begin{bmatrix}
			S_n & 0\\
			\textbf{D}_{S_n} & M_z 
		\end{bmatrix}\right).
	\end{equation*} 
	Moreover, if we assume that $\mathcal{H}_{Sc}$ is the model space given by $\mathcal{H}_{Sc}=\operatorname{Ran}\mathscr{E}_{Sc}$. Then equation \eqref{eq5.61} implies that $(S_1,\dots,S_n)$ is jointly equivalent to 
	$$	P_{\mathcal{H}_{Sc}}\left(\begin{bmatrix}
		S_1 & 0\\
		\textbf{F}_{n-1}^*\textbf{D}_{S_n} & M_{\Psi_1}
	\end{bmatrix},\dots,\begin{bmatrix}
		S_{n-1} & 0\\
		\textbf{F}_1^*\textbf{D}_{S_n} & M_{\Psi_{n-1}}
	\end{bmatrix},\begin{bmatrix}
		S_n & 0\\
		\textbf{D}_{S_n} & M_z 
	\end{bmatrix}\right)\Bigg|_{\mathcal{H}_{Sc}}.$$ 
	This completes the proof.
\end{proof}

\begin{rem}
	Let $(S_1,\dots,S_n)$ be a $\Gamma_n$-contraction acting on a Hilbert space $\mathcal{H}$, and let $(F_1,\dots,F_{n-1})$ and $(E_1,\dots,E_{n-1})$ be the fundamental operator tuples of $(S_1,\dots,S_n)$ and $(S_1^*,\dots,S_n^*)$  respectively. Suppose that $(E_1,\dots,E_{n-1})$ satisfies the commutativity condition \eqref{eq1.1}. Let $\Phi_1,\dots,\Phi_{n-1}\in H^{\infty}(\mathcal{B}(\mathcal{D}_{S_n^*}))$ be obtained as in Theorem \ref{thm2.3} such that the tuple $(\gamma_1 M_{\Phi_1}, \dots, \gamma_{n-1} M_{\Phi_{n-1}})$ is a $\Gamma_{n-1}$-contraction. If the tuple $(D_1,\dots,D_n)$ be the $\Gamma_n$-unitary canonically constructed using $(S_1,\dots,S_n)$. Then in view of equation \eqref{eq5.56}, the operator tuple $$\left( \begin{bmatrix}
		M_{\Phi_1} & 0\\
		0 & D_1
	\end{bmatrix},\dots,\begin{bmatrix}
		M_{\Phi_{n-1}} & 0\\
		0 & D_{n-1}
	\end{bmatrix},\begin{bmatrix}
		M_z & 0\\
		0 & D_n 
	\end{bmatrix} \right)$$
	is unitarily equivalent to  
	$$\left(\begin{bmatrix}
		S_1 & 0\\
		\textbf{F}_{n-1}^*\textbf{D}_{S_n} & M_{\Psi_1}
	\end{bmatrix},\dots,\begin{bmatrix}
		S_{n-1} & 0\\
		\textbf{F}_1^*\textbf{D}_{S_n} & M_{\Psi_{n-1}}
	\end{bmatrix},\begin{bmatrix}
		S_n & 0\\
		\textbf{D}_{S_n} & M_z 
	\end{bmatrix}\right).$$
\end{rem}
%\medskip \textit{Acknowledgment}. 
\textsl{Acknowledgements:}
	The work of the first author is supported through SERB-CRG (CRG/2022/003058).
%	 and the second named author thankfully acknowledges the financial support provided by  the research project of SERB with
%	ANRF File Number: CRG/2022/003058, by the Science and Engineering Research Board (SERB),
%	Department of Science and Technology (DST), Government of India.

	%	\vspace{.5cm}
	%	
	%	\vspace{.5cm}
	%	\noindent (Keshari) \sc{An OCC of Homi Bhabha National Institute, School of Mathematical Sciences, National Institute
		%		of Science Education and Research, Bhubaneswar, Post-Jatni, Khurda, 752050, India}\\
	%	{E-mail address:} {\bf dinesh@niser.ac.in}
\end{document}